\documentclass{aims}
\usepackage{amsmath}
  \usepackage{paralist}
  \usepackage{graphics} 
  \usepackage{epsfig} 
\usepackage{graphicx}  \usepackage{epstopdf}
 \usepackage[colorlinks=true]{hyperref}
\hypersetup{urlcolor=blue, citecolor=red}

\usepackage{amssymb, comment}

\usepackage{todonotes}

\usepackage{comment}

\def\currentvolume{X}
 \def\currentissue{X}
  \def\currentyear{200X}
   \def\currentmonth{XX}
    \def\ppages{X--XX}

\newtheorem{theorem}{Theorem}[section]

\newtheorem{lemma}[theorem]{Lemma}
\newtheorem{proposition}{Proposition}

\theoremstyle{definition}
\newtheorem{definition}[theorem]{Definition}

\title[Multi-agent Systems for Quadcopters] 
      {Multi-agent Systems for Quadcopters}

\author[Carney, Chyba, Gray, Shanbrom and Wilkens]{}


 \email{rscarney@hawaii.edu}
 \email{chyba@hawaii.edu}
 \email{cgray8@hawaii.edu}
 \email{grw@math.hawaii.edu}
 \email{corey.shanbrom@csus.edu}

\thanks{$^*$ Corresponding author: M. Chyba}

\begin{document}
\maketitle

\centerline{\scshape Richard Carney, Monique Chyba$^*$, Chris Gray, and George Wilkens }
\medskip
{\footnotesize
 \centerline{Department of Mathematics, University of Hawaii at Manoa}
   \centerline{2565 McCarthy Mall}
   \centerline{Honolulu, Hawaii 96822, USA}
} 

\medskip

\centerline{\scshape Corey Shanbrom}
\medskip
{\footnotesize
 \centerline{Department of Mathematics and Statistics}
   \centerline{California State University, Sacramento}
   \centerline{6000 J St., Sacramento, CA 95819, USA}
}

\bigskip


\begin{abstract}
Unmanned Aerial Vehicles (UAVs) have been increasingly used in the context of remote sensing missions such as target search and tracking, mapping, or surveillance monitoring. In the first part of our paper we consider agent dynamics, network topologies, and collective behaviors. The objective is to enable multiple UAVs to collaborate toward a common goal, as one would find in a remote sensing setting. An agreement protocol is carried out by the multi-agents using local information, and without external user input. The second part of the paper focuses on the equations of motion for a specific type of UAV, the quadcopter, and expresses them as an affine nonlinear control system. Finally, we illustrate our work with a simulation of an  agreement  protocol  for  dynamically  sound  quadcopters augmenting the particle graph theoretic approach with orientation and a proper dynamics for quadcopters. 
\end{abstract}


\section{Introduction}
Unmanned Aerial Vehicles (UAVs) present unique opportunities to investigate unsafe or otherwise unreachable locations. UAVs are already being used in numerous agricultural applications (\cite{Kim,Costa}), search and rescue (\cite{Doherty,Waharte}), and conservation biology (\cite{Wich}).  While there are still many open issues in the use of UAVs, they also have much to offer when it comes to disaster preparation and response (\cite{Erdelj}).

In Spring 2018, motivated by the 2018 Kilauea volcano eruption, we conducted work in collaboration with NASA Jet Propulsion Laboratory (JPL) to develop a real-time data relay system for UAVs.
A relay system augmented with an interactive visual aid for data viewing was developed and implemented to exploit the live stream, including image superposition between sequential flights to rapidly provide information about changes to the situation under study. In September 2018, we demonstrated the concept at the University of Hawai’i at M\=anoa with live streaming of images from Bellows Air Force Station and video from Fissure 8 on Hawai’i Island. We then developed on-board capability for image analysis for real time situation assessment (\cite{IGARSS}). In parallel, we also started to study motion planning for a swarm of UAVs (\cite{IGARSS2}), which is the problem we focus on in this paper. See \cite{Videodrones} for a video demonstrating some of the work.  

Implementing a coordinated group of UAVs, rather than a single vehicle, is naturally more efficient for many mission types. These include tracking, surveillance, and mapping  (\cite{AerialSwarm}). In this regard, there are numerous reasons, driven by the variety of applications, why collective control of a group of UAVs is important. Common objectives include synchronization, formation, flocking, consensus, and rendezvous. Rendezvous is the primary focus here (\cite{CollectiveControl}). Moreover, autonomy of individual UAVs provides a flexibility and adaptability that can be beneficial in many kinds of applications. Optimization techniques, as an approach to collective control, have the general drawback of scaling poorly with the number of vehicles. We shall focus on a common alternative: the so-called graph theoretic methods. The relative simplicity of these methods makes them an attractive choice in terms of both scalability and the desire for real-time control (\cite{Meshabi,PathPlanning}). Moreover, they apply quite naturally to situations where there is an existing network of communications between the UAVs, although the existence of such is also an additional assumption being made. The benefits include control that is decentralized in that it occurs over the network without any advanced planning, and certain objectives, such as rendezvous, can be carried out with only relative information being shared across the network.

In this paper we focus on path-planning of the most common type of UAV, quadcopters, and we focus on the dynamics of this specific type of UAV in Sections \ref{section-dynamics} and \ref{section-quadtrajectories}.  Section \ref{section multi-agents}, however, applies to any multi-agent UAV path configuration since we treat them as particle points. We also occasionally refer to any UAV by the common, though imprecise, more general term ``drone". These small versatile vehicles have proved to be a major asset in the increasing demand for complex autonomous robotics missions, particularly those involving survey and object detection. Our end-goal is developing, for a swarm of quadcopters, the ability to autonomously make decisions based on shared information. 

The outline of the paper is as follows. In Section \ref{section multi-agents} we discuss the multi-agent autonomous agreement. Agents are considered as particles and the methodology is based on graph theory. Section \ref{section-dynamics} introduces the dynamics of quadcopters and their navigation. In Section \ref{section-quadtrajectories} we approximate our particle multi-agent trajectories by implementable ones satisfying the quadcopter's dynamic equations of motion. 

It is a genuine pleasure to have our article appear in this special issue that honors Professor Bloch. The depth and breadth of his prodigious body of work inspires and enriches each of us. On those occasions when the authors feel like explorers venturing into uncharted territory, we often find a flag marking Tony's prior visitation. Indeed, Tony's contributions to the subject of geometric control of mechanical systems are keenly felt in this paper. His recent work on multi-agent formations with obstacle and inter-agent avoidance in \cite{VarCollision} aims toward a similar goal as ours. Although the mathematical approach taken differs from the one here, some of his work in \cite{QuantLind} does apply a graph theoretic methodology.  
In addition, the rigid body dynamics that we highlight in the second half of the paper is core to much of his research (\cite{Bloch book,Bloch1,GeomStructPreser,PlanarMotGrav}). With sincere appreciation and gratitude, the authors offer Professor Bloch their very best wishes.

\newpage
\section{Multi-Agent Autonomous Agreement}
\label{section multi-agents}
Suppose a mission with multiple agents requires that at some point the vehicles come together at a single location, for instance at the end of a mission.  With batteries possibly running low, a main objective is to conserve power and rely on short-range communication. Autonomous agreement is one of the fundamental problems in multi-agent coordination,
where a collection of agents are to agree on a joint state value.

\subsection{Agreement Protocol}
Our multi-agent systems will be viewed as networks, as described in this section. Our primary goal is to highlight the existence of an intricate relationship between the convergence properties of the so-called {agreement protocol} on one hand, and the structure of the underlying agent network connections on the other.
\begin{definition}
We define a {\it network} of agents to be a finite collection of  agents
possessing some inter-connectivity via relative information exchange links. The network is modeled by a graph, denoted
$G(V,E)$, where $V$ is the set of vertices and $E$ is the set of edges.
The agents correspond to the vertices of $G(V,E)$ and the relative information
exchange links correspond to the edges. 
\end{definition}
In the sequel we denote by $V=\{v_1,v_2,\dots,v_n\}$ 
the vertices corresponding to the $n$ agents, labeled $1,2,\dots,n$. The set of edges is a subset
$E\subset V\times V$; each edge has the form $(v_i,v_j)=v_iv_j$.
We assume $v_iv_j\in E$ implies $v_jv_i\in E$ and that $v_iv_i\notin E$.

We shall enhance this initial model to accommodate time-varying weighted networks, motivated by the desire to account for the strength of communications between agents, or possibly to take into account how the strength of communications evolves in the rendezvous problem as distances between agents change.  Interestingly, suitable assumptions on the weights imply the agreement value is an invariant, and only the trajectories leading to it differ.

The agreement protocol assumes the rate of change of each agent's state is governed by a sum of relative
states with respect to a subset of neighboring agents. Assume the scalar function $q^i(t)$ denotes the state of agent $i$ for $i=1,\dots,n$, and define $q(t)=(q^1(t),\cdots,q^n(t))^T$. It follows that   
\begin{equation}
\label{eqqdot1}
\dot{q}^i(t)=-\sum_{j\in N(i)}\big(q^i(t)-q^j(t)\big), \qquad 1\leq i\leq n,
\end{equation}
where $N(i)$ is the neighborhood of a vertex $v_i$, i.e., the collection of adjacent vertices $N(i)=\{v_j\in V \mid v_iv_j\in E\}$. If we adopt the notion that adjacency is symmetric, i.e., $v_iv_j\in E\iff v_jv_i\in E$, then we
say the network is {\it undirected} and we can
encapsulate the above system in a single matrix equation:
\begin{equation}
\label{eq-agreementdynamics}
\dot{q}(t)=-L(G)q(t),
\end{equation}
where $L(G)\in\mathbb{R}^{n\times n}$ is
called the {\it graph Laplacian}.
The graph Laplacian is uniquely determined by the structure of $G(V,E)$.
\begin{definition}
We refer to Equation (\ref{eq-agreementdynamics}) as the \emph{agreement dynamics}, the \emph{agreement protocol}, or the \emph{consensus protocol}.
\end{definition}
\subsection{Consensus Joint Value}
The graph Laplacian has desirable properties when
$G(V,E)$ is a connected, undirected network. In the rest of this section we will always assume $G(V,E)$ is a connected, undirected graph unless specified otherwise.
\begin{proposition}
Assume $G(V,E)$ is a connected, undirected graph.
The graph Laplacian $L(G)$ is a symmetric positive
semidefinite matrix with eigenvalues $0=\lambda_1<\lambda_2\leq\dots\leq\lambda_n$. Moreover, the $n$-dimensional vector of all ones, $\vec 1$,  is an eigenvector corresponding to the zero eigenvalue $\lambda_1$, i.e.,
$L(G)\vec 1=\vec 0$ .
\end{proposition}
\begin{proof}
This is a standard result and can be found in \cite{DecentralizedControl, PathPlanning, Consensus}.
\end{proof}
Let $\Lambda={\rm diag}(0,\lambda_2,\dots,\lambda_n)$ and let $U=[\hat{u}_1\ \hat{u}_2\ \dots \ \hat{u}_n]$ (the hat notation indicates a unit vector) be an $n\times n$ matrix consisting of orthonormal eigenvectors of $L(G)$
corresponding to the eigenvalue ordering above, i.e., $U^T\,U=I$ and $L(G)\,U=U\,\Lambda$. Since we can factor the Laplacian as $L(G)=U\Lambda U^T$, the solution to Equation (\ref{eq-agreementdynamics}) with initial condition $q(0)=q_0$ is given by $$q(t)=e^{-L(G)t}q_0=e^{-U\Lambda U^Tt}q_0=Ue^{-\Lambda t}U^Tq_0,$$ which can be written as:
\begin{equation}
\label{agreementsolution}
    q(t)=\sum_{i=1}^n e^{-\lambda_i t}(\hat{u}^{T}_i q_0)\hat{u}_i.
    \end{equation}
\begin{definition}
Given a set of $n$ multi-agents, its
\emph{agreement set} $\mathcal{A}\subset\mathbb{R}^n$ is defined as the subspace $span\{\vec{1}\}$. A \emph{consensus joint state value} is an element of $\mathcal{A}$, and the common value of its (identical) components is called its \emph{agreed state value}. 
\end{definition}
The next proposition shows that when the state of a network of multi-agents satisfies the agreement protocol, the state converges to the consensus joint state value whose agreed state value is obtained by averaging the state's initial components.
\begin{proposition}
Let $G(V,E)$ be a network of $n$ agents whose state $q(t)$ satisfies the agreement protocol $\dot{q}=-L(G)q$. Then
$\lim_{t\rightarrow\infty} q(t)=\alpha \vec 1$, where
$\alpha=\frac{\vec 1^{T}q(0)}{n}\in\mathbb{R}$ is the average of the initial states. In addition, since $\lambda_2$ is $L(G)$'s smallest nonzero eigenvalue, its magnitude will determine the rate of convergence to $\alpha\vec 1$ 
(\cite{Consensus, DecentralizedControl}). 
\end{proposition}
\begin{proof}
Since $\lambda_1=0$ and $\lambda_i>0$ for $i\geq 2$ we have,
from Equation (\ref{agreementsolution}), that
$q(t)$ converges to the consensus joint state value 
$(\hat{u}^{T}_1 q_0)\hat{u}_1=\frac{\vec 1^{T}q_0}{n}\vec 1=\alpha\vec 1\in \mathcal{A}$ as $t\rightarrow\infty$,
where $\frac{\vec 1^{T}q_0}{n}=\alpha\in\mathbb{R}$ is its agreed state value. Moreover, Equation (\ref{agreementsolution}) implies that
$e^{\lambda_2 t}(q^i(t)-\alpha)$ converges, as $t\rightarrow\infty$, for all $1\leq i\leq n$.
\end{proof}
The point $\displaystyle\lim_{t\rightarrow\infty}q(t)=\frac{\vec 1^{T}q_0}{n}\vec 1=\alpha\vec 1\in\mathbb{R}^n$ is precisely
the orthogonal projection of $q_0$ onto the agreement subspace
and therefore minimizes the quantity $\lVert q-q_0\rVert$ over all possible $q\in \mathcal{A}$,
where $\lVert \cdot\rVert$ is the standard Euclidean distance.
Furthermore, $\frac{\vec 1^{T}q(t)}{n}$ is a constant of motion since
\begin{align}
\frac{d}{dt}\left(\frac{\vec 1^{T}q(t)}{n}\right)= & \frac{\vec 1^{T}\dot{q}(t)}{n}=\frac{\vec 1^{T}(-L(G)q(t))}{n}
=-\frac{q^T(t)L^T(G)\vec{1}}{n}\\= & -\frac{q^T(t)L(G)\vec{1}}{n}=-\frac{q^T(t)\vec{0}}{n}=0
\end{align}
for arbitrary values of $t$. Hence the agreed state value $\alpha=\frac{\vec 1^{T}q_0}{n}$ is completely determined at $t=0$. 

If the state of each agent is a vector quantity rather than a scalar value, the agreement protocols for the components can still be written in a single matrix equation. Indeed, assume we desire agreement on $r$ scalar values. This means that to each agent $i$ we associate a vector value $q^i(t)=(q_1^i(t),\cdots,q_r^i(t))$ and we let $Q$ be the $n\times r$ matrix whose rows are given by the $q^i$,
\begin{equation}
    Q=\left[\begin{array}{ccc}
    q_1^1&\cdots&q_r^1\\
    q_1^2&\cdots&q_r^2\\
    \vdots & \vdots& \vdots\\
    q_1^n&\cdots&q_r^n\\
    \end{array}\right]
\end{equation}
and we can write the agreement dynamics as:

\begin{equation}
\label{eqdiff-laplacian}
    \dot Q(t)=-L(G)Q(t).
\end{equation}
We introduce $q^i(0)=(q^i_1(0),\cdots, q^i_r(0))$ the initial row vector for the state of agent $i$ and, more importantly, $q_k(0)=(q_k^1(0),\cdots, q_k^n(0))^T$ the initial column vector for a given state value over all agents. Since $L(G)=U\Lambda U^T$, we have that the solution of (\ref{eqdiff-laplacian}) is given by:
$Q(t)=e^{-L(G)t}Q(0)=e^{-U\Lambda U^Tt}Q(0)=Ue^{-\Lambda t}U^TQ(0)$, which can be written as:
\begin{equation}
    Q(t)=\Bigg[\sum_{i=1}^n e^{-\lambda_it}\hat{u}_i^Tq_1(0)\hat{u}_i\,\,\cdots\,\,\sum_{i=1}^n e^{-\lambda_it}\hat{u}_i^Tq_r(0)\hat{u}_i\Bigg]
\end{equation}
$$
    =[e^{-\lambda_1t}\hat{u}_1^Tq_1(0)\hat{u}_1\,\,\cdots\,\,e^{-\lambda_1t}\hat{u}_1^Tq_r(0)\hat{u}_1]+\sum_{i=2}^n[e^{-\lambda_it}\hat{u}_i^Tq_1(0)\hat{u}_i\,\,\cdots\,\, e^{-\lambda_it}\hat{u}_i^Tq_r(0)\hat{u}_i]
$$
$$
    =\Bigg[\frac{\vec{1}^Tq_1(0)}{n}\vec{1}\,\,\cdots\,\,\frac{\vec{1}^Tq_r(0)}{n}\vec{1}\Bigg]+\sum_{i=2}^n[e^{-\lambda_it}\hat{u}_i^Tq_1(0)\hat{u}_i\,\,\cdots\,\, e^{-\lambda_it}\hat{u}_i^Tq_r(0)\hat{u}_i]
$$
\begin{equation}
    =\Big[\alpha_1\vec{1}\,\,\cdots\,\,\alpha_r\vec{1}\Big]+\sum_{i=2}^n[e^{-\lambda_it}\hat{u}_i^Tq_1(0)\hat{u}_i\,\,\cdots\,\, e^{-\lambda_it}\hat{u}_i^Tq_r(0)\hat{u}_i]
\end{equation}
where $\frac{\vec{1}^Tq_k(0)}{n}=\alpha_k\in\mathbb{R}$ is the
agreed state value for each scalar state, $1\leq k\leq r$.

Note that here we assume the states $q_k$, measured by the agents, are shared over the same network $G(V,E)$. Hence the states share a common graph Laplacian, $L(G)$. In more general situations each state variable could have a distinct network,  leading to an agreement protocol of the form: $\dot{x}(t)=-L(G_x)x(t),\dots,\dot{z}(t)=-L(G_z)z(t)$.
\subsection{Example:}
Consider a collection of $n$ agents required to meet at a single location, not given in advance, and the agents do not have access to their global positions. Rather, all they can measure is their relative
distance with respect to their neighbors and they have to agree on their spatial coordinates $x,y$ and $z$. In this case, $r=3$ and $q^i(t)=(q^i_1(t),q^i_2(t),q^i_3(t))^T=(x^i(t),y^i(t),z^i(t))^T\in\mathbb{R}^3$ with $Q\in\mathbb{R}^{n\times 3}$ as an example of the agreement protocol in action. By executing the agreement protocol, the convergence to a unique rendezvous point $(\alpha_x,\alpha_y,\alpha_z)$ is guaranteed
for any arbitrary initial conditions: $(x^1(0),y^1(0),z^1(0)),\dots,(x^n(0),y^n(0),z^n(0))$ (i.e., any initial spatial configuration for the $n$ agents)
so long as $G(V,E)$ is a connected graph.
\subsection{Multi-agent Trajectories}
In this section, $\hat{e}_j$ denotes the standard unit vector of all $0$'s
with a single $1$ in the $j$th component. We use $u_1$ to denote the vector $\vec 1$, which is an eigenvector of $L(G)$ with eigenvalue $\lambda_1=0$. With this notation, $\hat{u}_1=\dfrac{\vec 1}{\sqrt{n}}$.
\begin{lemma}
\label{lemma1-th1}
Assume agent $j$ is connected to every other agent, and set $u_1=\vec 1$. Then, $u_1-n\hat{e}_j$ is an eigenvector for $L(G)$ with corresponding eigenvalue $\lambda_n=n$:
\begin{equation}
    L(G)(u_1-n\hat{e}_j)=n(u_1-n\hat{e}_j).
\end{equation}
\end{lemma}
\begin{proof}
We have $L(G)(u_1-n\hat{e}_j)=L(G)u_1-nL(G)\hat{e}_j=0-nL(G)_j=-nL(G)_j$, where $L(G)_j$ is the $j$th column of the Laplacian $L(G)$.
But since $G(V,E)$ has the $j$th node with $n-1$ connections
we know $L(G)_j=(-1,-1,\dots,n-1,\dots,-1,-1)^T=n \hat{e}_j-u_1$, i.e.,
all $-1$'s except an $n-1$ in the $j$th component.
Hence $-n L(G)_j=n(u_1 - n\hat{e}_j)$
as claimed.

\end{proof}

\begin{definition}
We say that vertex $v_i$ is \emph{fully connected} if there exists an edge between $v_i$ and every other vertex $v_j\in V, j\neq i$. Therefore a fully connected vertex has $n-1$ edges. 
\end{definition}
Recall that $\hat{u}_1=\frac{\vec 1}{\sqrt{n}}$ and $\lambda_1=0$.
Furthermore, it is proven in \cite{Eigenvalues} that the eigenvalues of $L(G)$ are less than or equal to $n$ and the multiplicity of the eigenvalue $\lambda=n$ is equal to the number of fully connected vertices.
\begin{theorem}\label{thm: connected graph}
Suppose vertex $v_j$ is fully connected and that $q(t)$ satisfies the agreement protocol $\dot{q}=-L(G)q$. Then $q^j(t)$ is a straight line. 
\end{theorem}
\begin{proof}
We know that $\displaystyle\sum_{i=1}^n q^i=n\alpha$ is a constant. We also know that:
\begin{equation}
\dot{q}^j=-n\,q^j+\sum q^i=-n(q^j-\alpha).
\end{equation}
This simple ODE (Newton's Law of Cooling) implies that \begin{equation}
    q^j-\alpha=(q^j(0)-\alpha)e^{-nt}
\end{equation} which is a straight line.
\end{proof}
Note that the proof is the same regardless of whether $q^j$ is a scalar or row $j$ of an $n\times r$ state matrix $Q(t)$, provided it satisfies the agreement protocol $\dot{Q}=-L(G) Q$.

\subsubsection{Rendezvous Mission, Unweighted Network}
\label{sectionredezvous1}
In this section we illustrate some of our prior results with a rendezvous mission between four agents.  We let $q^i(t)=(x^i(t),y^i(t),z^i(t))$ denote the position of agent $i$ in $\mathbb{R}^3$. We assume the agents to be initially distributed as follows:
\begin{equation*}
    q^1(0)=(4,17,24),\;q^2(0)=(18,10,32),\; q^3(0)=(15,10,26),\;q^4(0)=(4,2,35).
\end{equation*}
In this simulation we assume the communication network between the agents is fixed and unweighted. From our algorithm, the vector of agreed state values (the agreed rendezvous position) is given by:
\begin{equation}
    q^*=(10.25,9.75,29.25).
\end{equation}
In Figure \ref{Fig-rendezvous1} we illustrate how two different communication networks produce different trajectories, even though the agreed rendezvous position is the same for each network since it depends only on the initial spatial positions of the agents. In the first scenario (solid curve), we assume the following: agent 1 is connected to agent 2; agent 2 is connected to agents 1, 3 and 4; agent 3 is connected to agents 2, 4; agent 4 is connected to agents 2, 3. This is represented by the following graph Laplacian:
\begin{equation}
    L_1=\left(\begin{array}{cccc}1&-1&0&0\\-1&3&-1&-1\\0&-1&2&-1\\0&-1&-1&2\end{array}\right).
\end{equation}
In the second scenario (dashed curve), we assume: agent 1 is connected to agents 3, 4; agent 2 is connected to agent 3; agent 3 is connected to agents 1, 2; agent 4 is connected to agent 1:
\begin{equation}
    L_2=\left(\begin{array}{cccc}2&0&-1&-1\\0&1&-1&0\\-1&-1&2&0\\-1&0&0&1\end{array}\right).
\end{equation}
In the first scenario agent 2 moves along a straight line since it is connected to all other agents. 
\begin{figure}[h!]
\includegraphics[width=0.7\linewidth]{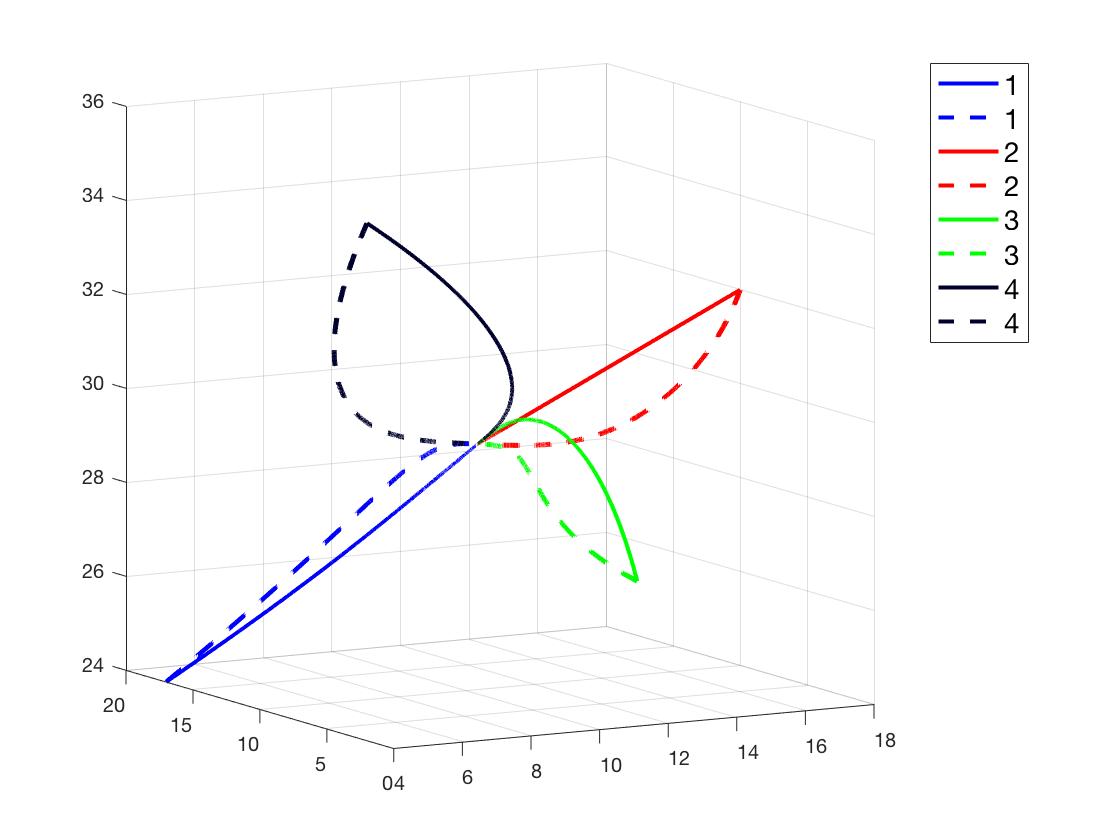}
\caption{Rendezvous Missions with Unweighted Network. Displays two different communication network scenarios for a 4-agent rendezvous mission. Agreement positions coincide but trajectories differ.}
\label{Fig-rendezvous1}
\end{figure}
In Figure \ref{Fig-rendezvous1-xyz} we compare the trajectories for agent 1 for each spatial coordinate. The eigenvalues for scenario 1 are given by $\{0,1,3,4\}$ and the eigenvalues for scenario 2 are given by $\{0,0.586,2,3.414\}$.

\begin{figure}[h!]
\includegraphics[width=.32\linewidth]{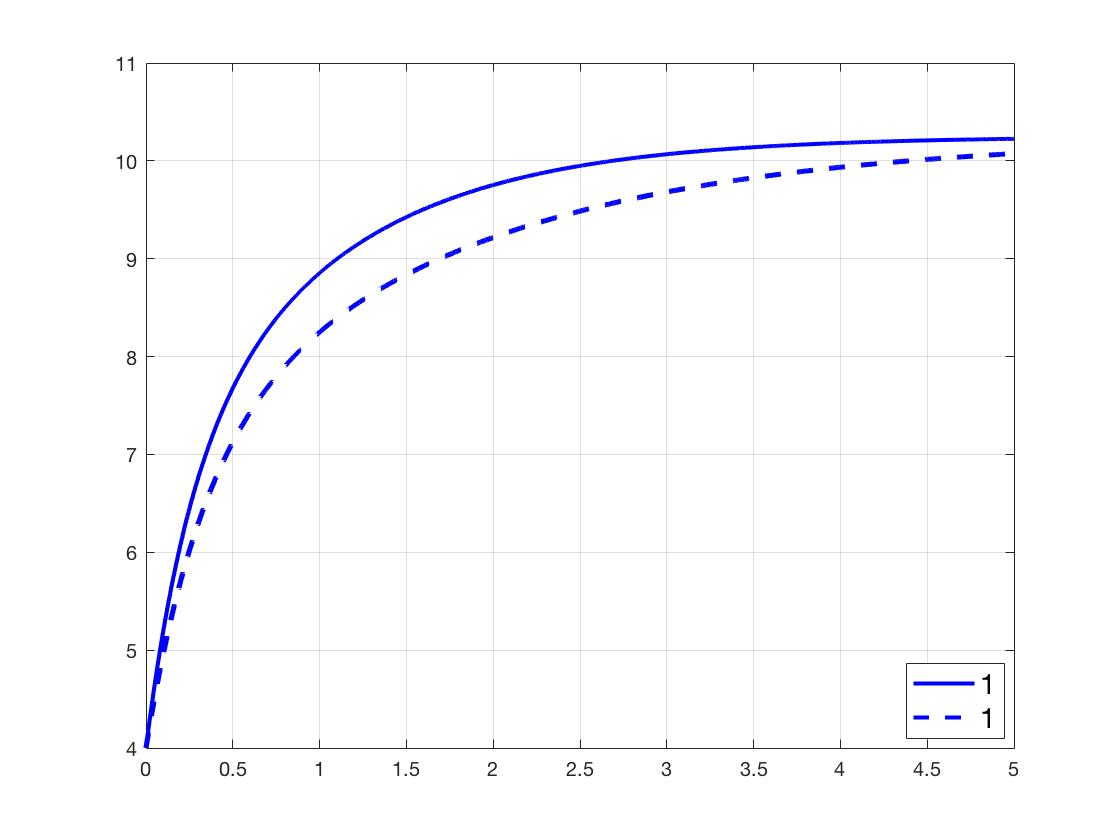}
\includegraphics[width=.32\linewidth]{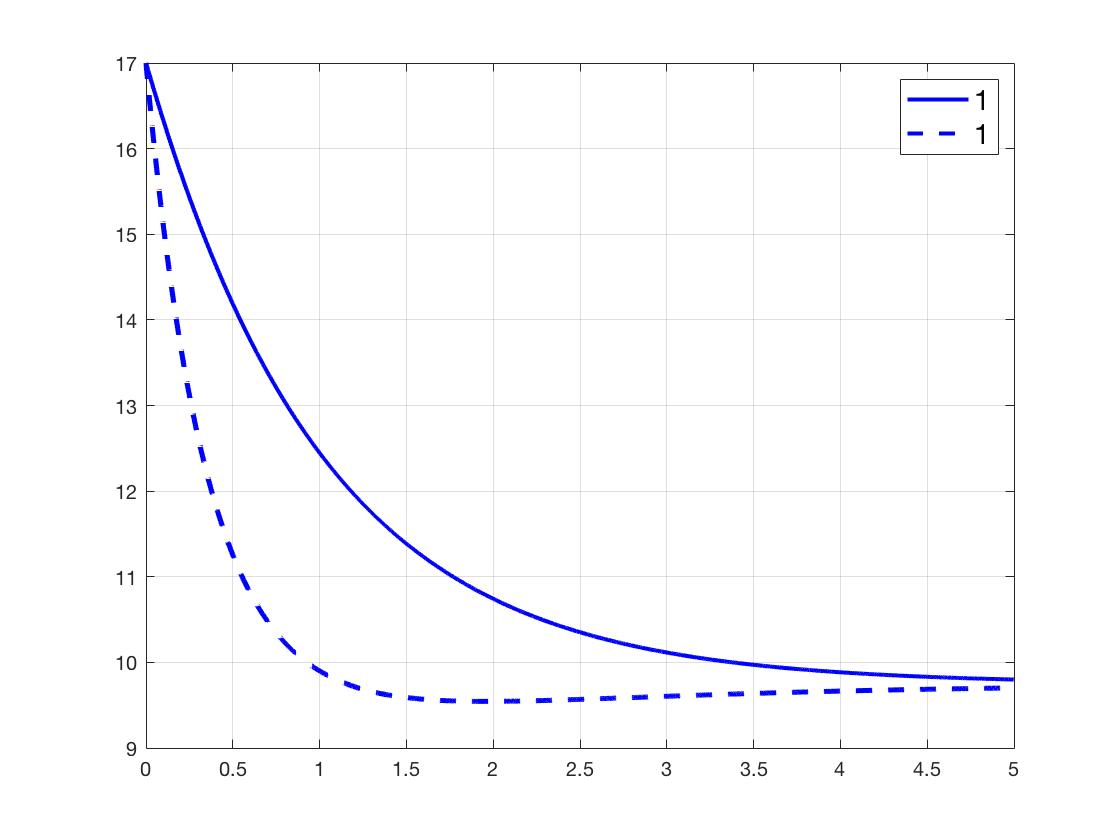}
\includegraphics[width=.32\linewidth]{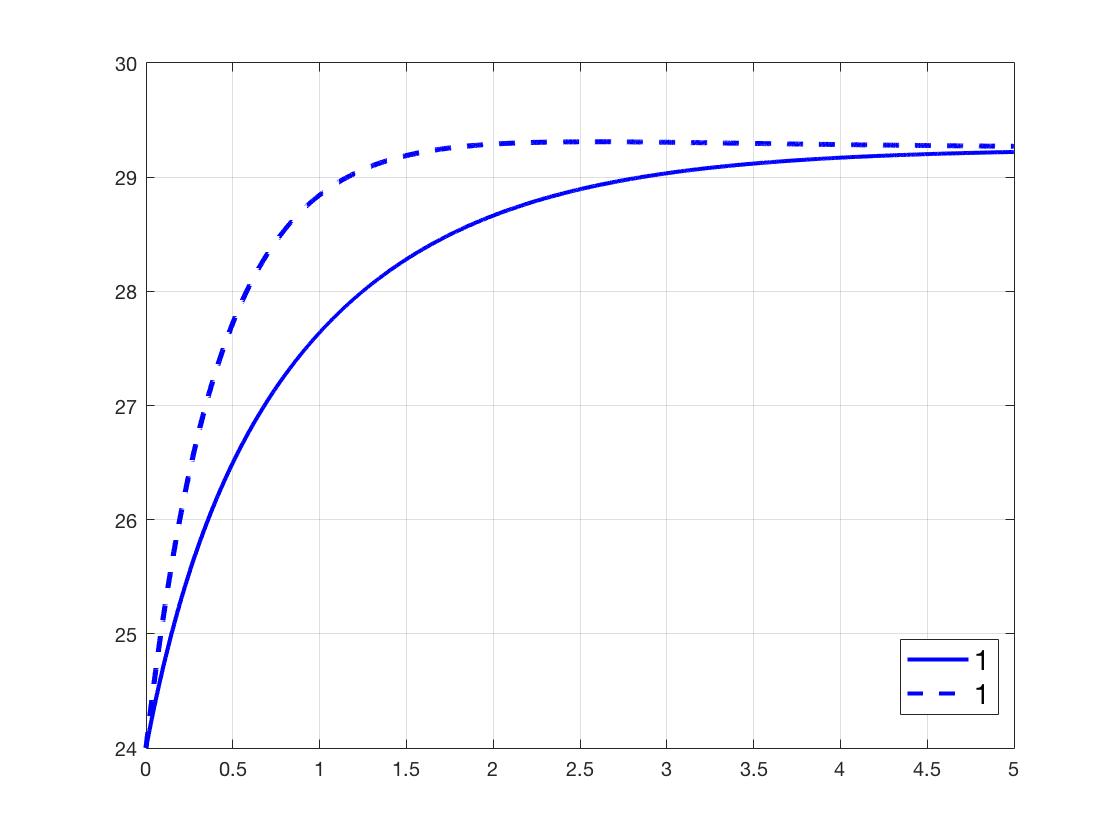}
\caption{Comparison of the $x,y,z$-motions for agent 1 for the scenarios of Fig. \ref{Fig-rendezvous1} corresponding to rendezvous missions with Unweighted Network.}
\label{Fig-rendezvous1-xyz}
\end{figure}

\subsection{Weighted Network}
We can imagine a scenario where it is desirable to account for the  communications strength between vehicles, or for how that strength evolves with time. Interestingly, the agreed state value is an invariant, and only the trajectories leading to it are affected \cite{Consensus}. Note that not all extensions are immediately straightforward, however, since the explicit solution may be difficult to write down. Not all properties can be expected to carry over; for example, the linear trajectory proof relied on cancellation due to the integer eigenvalues.

We now assume that the edges are weighted, i.e., edges carry a numerical weight. If the graph $G_w$ is a weighted graph, then the off-diagonal entries of the Laplacian $L(G_w)$ will not be limited to only $0$ and $-1$. Instead, the $-1$'s corresponding to edges in the graph would be replaced with $-w_{ij}=-w_{ji}$ to account for the weights. The diagonal entries will once again be the absolute value of the row or column sum of the non-diagonal entries. Equations (\ref{eqqdot1}) and (\ref{eq-agreementdynamics}) therefore become:
\begin{align}
    \dot{q}^i(t)&=-\sum_{j\in N(i)}w_{ij}\big(q^i(t)-q^j(t)\big)\\
    \dot q(t)&=-L(G_w)q(t).
\end{align}
 The proof that the  consensus  joint  state  value is given by  $\alpha=\frac{\vec{1}q_0}{n}\vec{1}$ comes from the same logic as before since $L(G_w)$ is a symmetric positive
semidefinite matrix with eigenvalues $0=\lambda_1<\lambda_2\leq\dots\leq\lambda_n$ and $L(G_w)\vec 1=\vec 0$. This is a consequence of the fact that weighting preserves the properties that $G_w(V,E)$ is connected and undirected. Moreover, 
each row and column of $L(G_w)$ still sums to zero regardless of whether or not the edges of $G_w(V,E)$ are weighted. Hence, the values of $\lambda_2,\dots\lambda_n$ and $\hat{u}_2\dots \hat{u}_n$, will change, but we still have $\lambda_1=0$ and $\hat{u}_1=\frac{\vec{1}}{\sqrt{n}}$. Finally, in full generality:

\begin{theorem}
Consider the dynamical system:
\begin{equation}
    \dot q^i(t)=-\sum_{j=1}^n w_{ij}(t)[q^i(t)-q^j(t)]\quad(1\leq i\leq n)
\end{equation}
with initial conditions $q(0)=(q^1(0),\cdots q^n(0))^T$.

For each $i,j$, assume $w_{ij}:[0,\infty)\to[0,\infty)$ and $w_{ij}(t)=w_{ji}(t)$.\\

Define the Laplacian matrix $L(G_w)(t)$ by
\begin{equation}
L(G_w)(t)=\left[\begin{array}{cccc}
    \sum_{j=1}^nw_{1j}(t)&-w_{12}(t)&\dots&-w_{1n}(t)\\
    -w_{12}(t)&\sum_{j=1}^nw_{2j}(t)&\dots&-w_{2n}(t)\\
    \vdots&\vdots&\vdots&\vdots\\
    -w_{1n}(t)&-w_{2n}(t)&\dots&\sum_{j=1}^nw_{nj}(t)\\
    \end{array}\right].
\end{equation}
Then the system of first order equations can be rewritten as: $\dot{q}(t)=-L(G_w)(t)q(t)$.
If the rank of $L(G_w)(t)$ is $n-1$ for all $t$ then we have:
$$\lim_{t\rightarrow\infty}q(t)=\bigg[\frac{1}{n}\sum_{i=1}^nq^i(0)\bigg]\vec{1}.$$
\end{theorem}

\begin{proof}
By assumption, the matrix $L(G_w)$ is real and symmetric, therefore Hermitian.
Furthermore, it is diagonally dominant with non-negative diagonal entries.
Therefore it is positive semi-definite. As for the unweighted case, 
$\lambda_1=0$ with corresponding eigenvector $\vec{1}$. Since we assume the graph to be connected throughout the motion, the rank of $L$ is always $n-1$ and the multiplicity of the zero eigenvalue is one. We have for any time $t$:  $0=\lambda_1<\lambda_2(t)\leq\dots\leq\lambda_n(t)$
with a corresponding set of orthonormal
eigenvectors $\hat{u}_1=\frac{1}{\sqrt{n}}\vec{1},\hat{u}_2(t),\dots,\hat{u}_n(t)$.

Let $q^*=\bigg[\frac{1}{n}\sum_{i=1}^nq^i(0)\bigg]\vec{1}$.
Even if the $w_{ij}:[0,\infty)\to[0,\infty)$ are time-varying,
we can show with initial conditions $q(0)$ that the trajectories $q(t)\to q^*$ as $t\to\infty$.
A useful observation is that
$$\dot{q}(t)=-L(G_w)(t)q(t)\quad\Rightarrow\quad q(t)-q(0)=-\int_{0}^{t}L(G_w)(s)q(s)ds$$
$$\Rightarrow\quad\vec{1}^T\big(q(t)-q(0)\big)=-\vec{1}^T\bigg(\int_{0}^{t}L(G_w)(s)q(s)ds\bigg)$$
$$=-\int_{0}^{t}\vec{1}^TL(s)q(s)ds=-\int_{0}^{t}\vec{0}^Tq(s)ds=-\int_{0}^{t}0ds=0,$$
so that we have $\vec{1}^Tq(t)=\vec{1}^Tq(0)$ for all $t\in[0,\infty)$.

Let $V:\mathbb{R}^n\to\mathbb{R}$ be defined by:
$$q\mapsto V(q)=\frac{1}{2}\|q-q^*\|^2=\frac{1}{2}(q-q^*)^T(q-q^*).$$
Clearly, we have:
\begin{enumerate}
    \item $V(q)\geq 0$\label{i},
    \item $V(q)=0\iff q=q^*$\label{ii},
    \item $\dot{V}(q)=-(q-q^*)^TL(G_w)(q-q^*)\leq 0$\label{iii}
\end{enumerate}
since $-v^TL(G_w)v\leq 0$ for any $v\in\mathbb{R}^n$ since $L(G_w)$ is positive semi-definite.
Clearly $q=q^*$ implies $\dot{V}(q)=0$, but also $\dot{V}(q)=0$ implies $q=q^*$.
Let $\mathcal{A}=span\{\vec{1}\}=\mathcal{N}(L(G_w))$ represent the null space
of $L(G_w)$ and suppose $(q-q^*)^TL(G_w)(q-q^*)=0$.
This implies $q-q^*\in\mathcal{A}$.
But we know that $\vec{1}^Tq(t)=\vec{1}^Tq(0)$, and therefore $q-q^*\perp\mathcal{A}$ since
\begin{equation}
\vec{1}^T(q-q^*)=\vec{1}^T\big(q-\frac{\vec{1}^Tq(0)}{n}\vec{1}\big)=\vec{1}^Tq-\frac{\vec{1}^Tq(0)}{n}\vec{1}^T\vec{1}=\vec{1}^Tq-\vec{1}^Tq(0)=0.
\end{equation}
Finally, $q-q^*\in\mathcal{A}$ and $q-q^*\perp\mathcal{A}$ imply $q-q^*=0\Rightarrow q=q^*$. We conclude:
\begin{equation}
\label{iV}
    \dot{V}(q)=-(q-q^*)^TL(G_w)(q-q^*)=0\iff q=q^*.
\end{equation}
Therefore \ref{i}, \ref{ii}, \ref{iii} and Equation (\ref{iV}) imply that $V(q)$ is a non-negative function which is strictly decreasing for all trajectories $q(t)\neq q^*$
hence $q(t)\to q^*$ as $t\to\infty$, i.e.,
$$\lim_{t\rightarrow\infty}q(t)=\bigg[\frac{1}{n}\sum_{i=1}^nq_i(0)\bigg]\vec{1}$$
as claimed.

\end{proof}

\subsubsection{Rendezvous Mission, Weighted Network}
 Here we imagine a scenario taking into account the strength of communications between agents. As previously stated, the agreement value is an invariant, and only the trajectories leading to it will differ compared to not taking the strength of the signal into account. 
We assume four agents to be initially distributed as follows:
\begin{equation*}
    q^1(0)=(2,16,29)   ,\;q^2(0)=(14,20,31),\; q^3(0)=(10,1,25),\;q^4(0)=(10,14,36).
\end{equation*}
The agreement rendezvous position is given by:
\begin{equation}
    q^*=(9,12.75,30.25).
\end{equation}
In Figure \ref{Fig-rendezvous2} we illustrate how a weighted network impacts the trajectories. For both scenarios we assume: agent 1 is connected to agent 2; agent 2 is connected to agents 1, 3 and 4; agent 3 is connected to agents 2, 4; agent 4 is connected to agents 2, 3 (as in Section \ref{sectionredezvous1}):
\begin{equation}
    L_1=\left(\begin{array}{cccc}1&-1&0&0\\-1&3&-1&-1\\0&-1&2&-1\\0&-1&-1&2\end{array}\right).
\end{equation}
In the second scenario (dashed curves), we assume a weighted network with weights due to initial distance:
\begin{equation}
    L_2=\left(\begin{array}{cccc}12.8062&-12.8062&0&0\\-12.8062&41.9036&-20.3224&-8.7750\\0&-20.3224&37.3518&-17.0294\\0&-8.7750&-17.0294&25.8044\end{array}\right).
\end{equation}
It can be observed that for the weighted network, even though agent 2 is connected to all other agents, it does not move along a straight line anymore. 

\begin{figure}[h!]
\includegraphics[width=.7\linewidth]{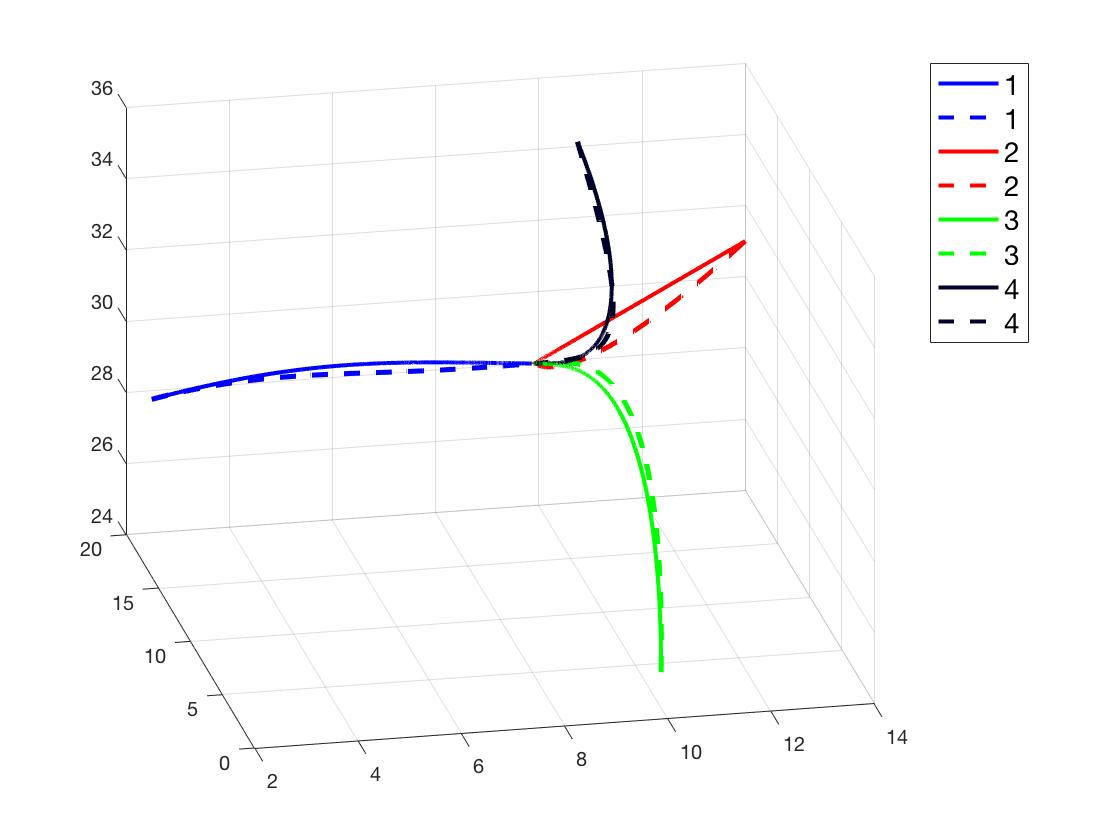}
\caption{Comparison between trajectories on rendezvous missions with unweighted  and weighted networks. The solid curves represents the trajectories for the unweighted network and the dashed ones for the weighted network. The two scenarios converge to the same agreement value.}
\label{Fig-rendezvous2}
\end{figure}

In Figure \ref{Fig-rendezvous2-xyz} we compare the trajectories for agent 1 for each of the coordinates of motion. The eigenvalues for the unweighted network are given by $\{0,1,3,4\}$ and the eigenvalues for the weighted network are given by $\{0,13.060,42.248,62.558\}$.

\begin{figure}[h!]
\includegraphics[width=.325\linewidth]{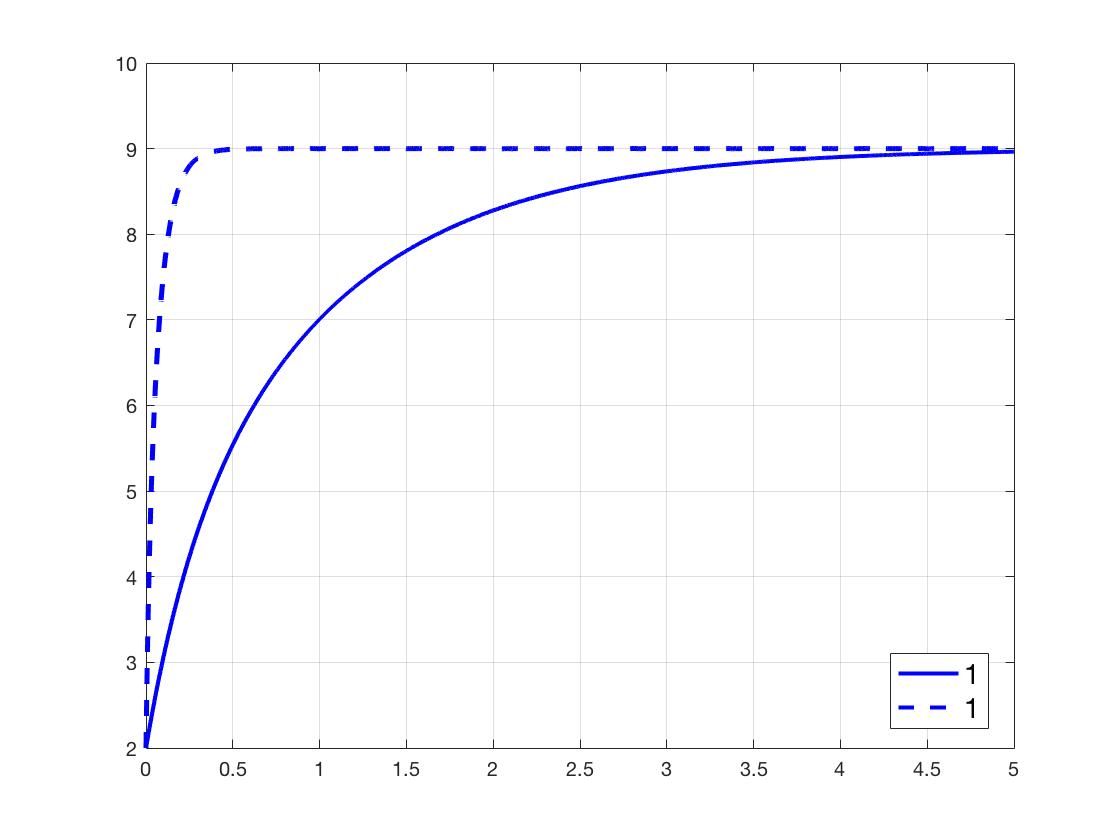}
\includegraphics[width=.325\linewidth]{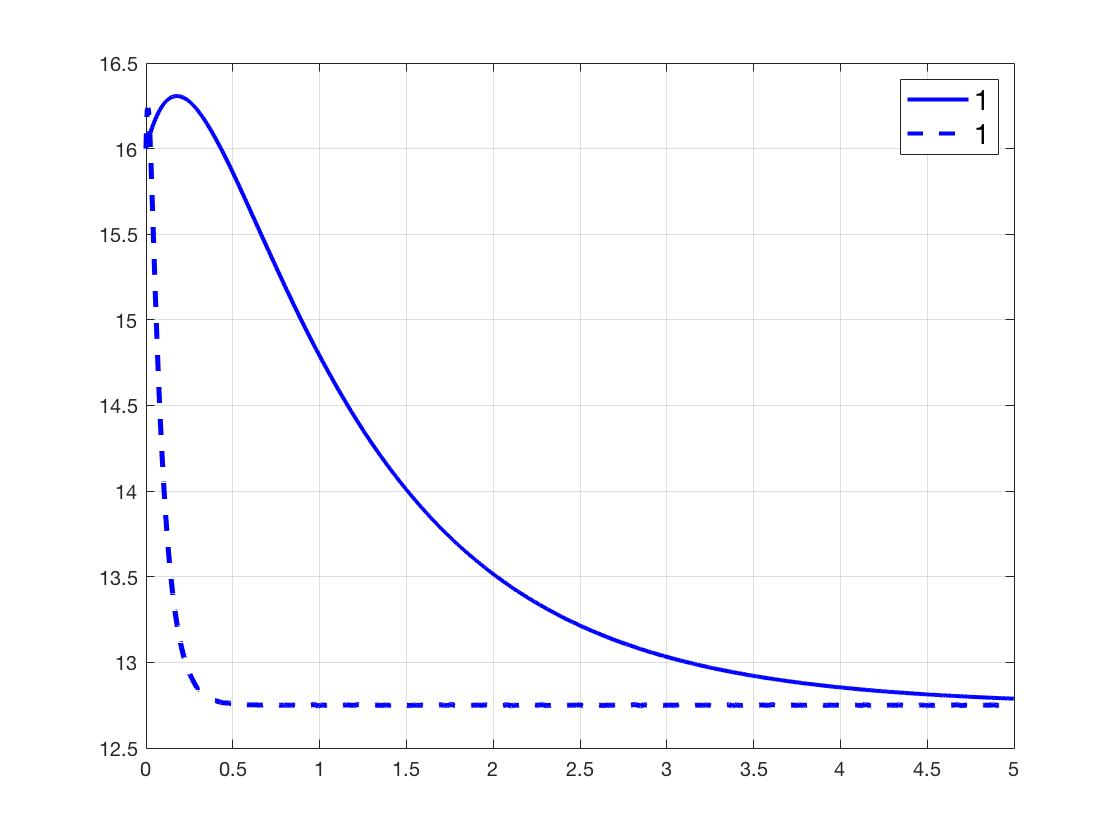}
\includegraphics[width=.325\linewidth]{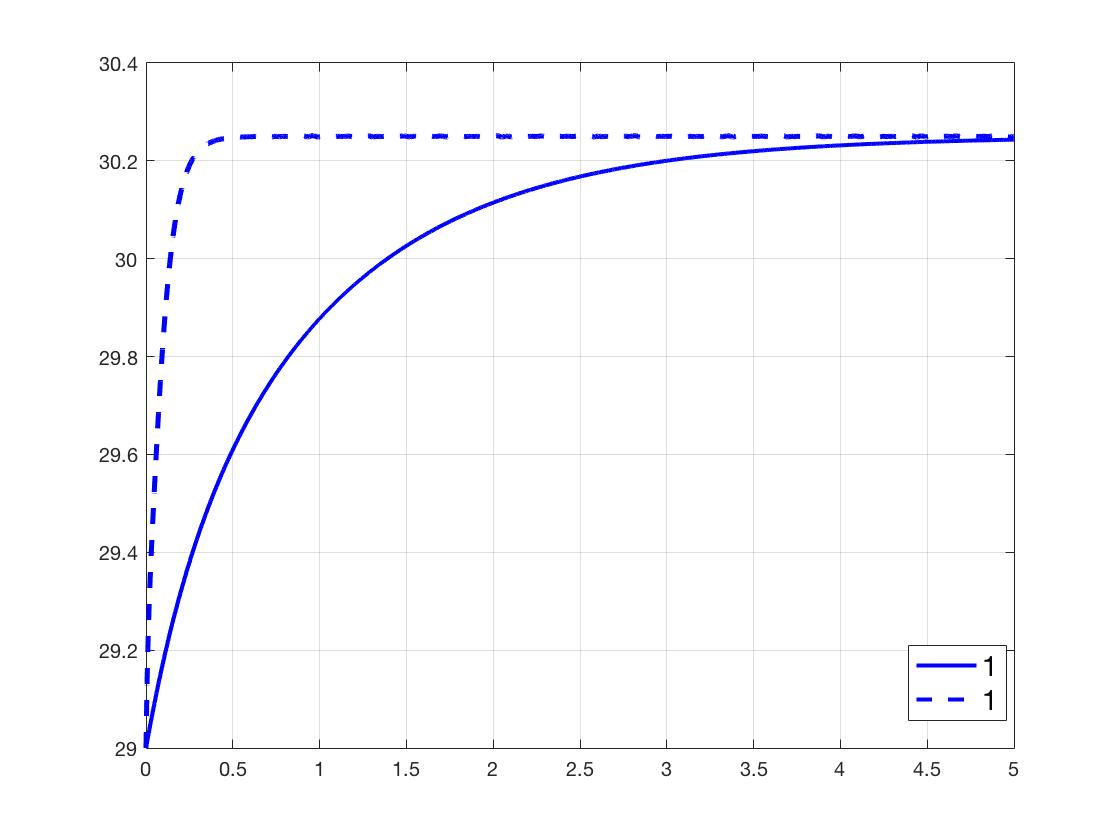}
\caption{Comparison of the $x,y,z$-motions for agent 1 for the rendezvous missions displayed in Figure \ref{Fig-rendezvous2}. Observe that each component converges much more rapidly for the weighted network, corresponding to the difference in their $\lambda_2$ values: $13.060>1$.}
\label{Fig-rendezvous2-xyz}
\end{figure}

\subsubsection{Rendezvous Mission, Time-varying Weighted Network}
The goal of this rendezvous mission is to compare a fixed network determined by initial proximity versus a time-varying weighted network, where the weights are the inter-agent distances that vary in time, and also edges are added as the preset proximity  threshold is reached. 

We assume four agents to be initially distributed as follows:
\begin{equation*}
    q^1(0)=(16,5,36)      ,\;q^2(0)=(19,19,29),\; q^3(0)=(12,16,33),\;q^4(0)=(14,1,26).
\end{equation*}
The agreement rendezvous position is given by:
\begin{equation}
    q^*=(15.25,10.25,31).
\end{equation}
In Figure \ref{Fig-rendezvous3} we illustrate how a weighted network impacts the trajectories. For the fixed network (solid curves) we have: agent 1 is connected to agents 3, 4; agent 2 is connected to agent 3; agent 3 is connected to agents 1, 3; agent 4 is connected to agent 1:
\begin{equation}
    L_1=\left(\begin{array}{cccc}2&0&-1&-1\\0&1&-1&0\\-1&-1&2&0\\-1&0&0&1\end{array}\right).
\end{equation}
In the time-varying scenario (dashed lines), the network changes throughout the trajectories. Here we provide the initial Laplacian and the final one:
\begin{equation}
    L_{2_i}=\left(\begin{array}{cccc}23.0375&0&-12.0830&-10.9545\\0&8.6023&-8.6023&0\\-12.0830&-8.6023&20.6854&0\\-10.9545&0&0&10.9545\end{array}\right),
    \end{equation}
\begin{equation}
    L_{2_f}=\left(\begin{array}{cccc}0.1465&-0.0523&-0.0482&-0.0460\\-0.0523&0.1501&-0.0450&-0.0528\\-0.0482&-0.0450&0.1457&-0.0525\\-0.0460&-0.0528&-0.0525&0.1513\end{array}\right).
\end{equation}
Initially the connectivity network is the same as for the fixed network but with weights, however as the agents get closer it changes, and all agents are connected in the end. The trajectories corresponding to the time-varying network here would be easier to approximate for a real quadcopter since they are ``straighter" (the sudden jump in information leads to straighter paths toward the rendezvous, however in this case, the weights prevent the perfect cancellation that would lead to true linearity.) 
\begin{figure}[h!]
\includegraphics[width=0.7\linewidth]{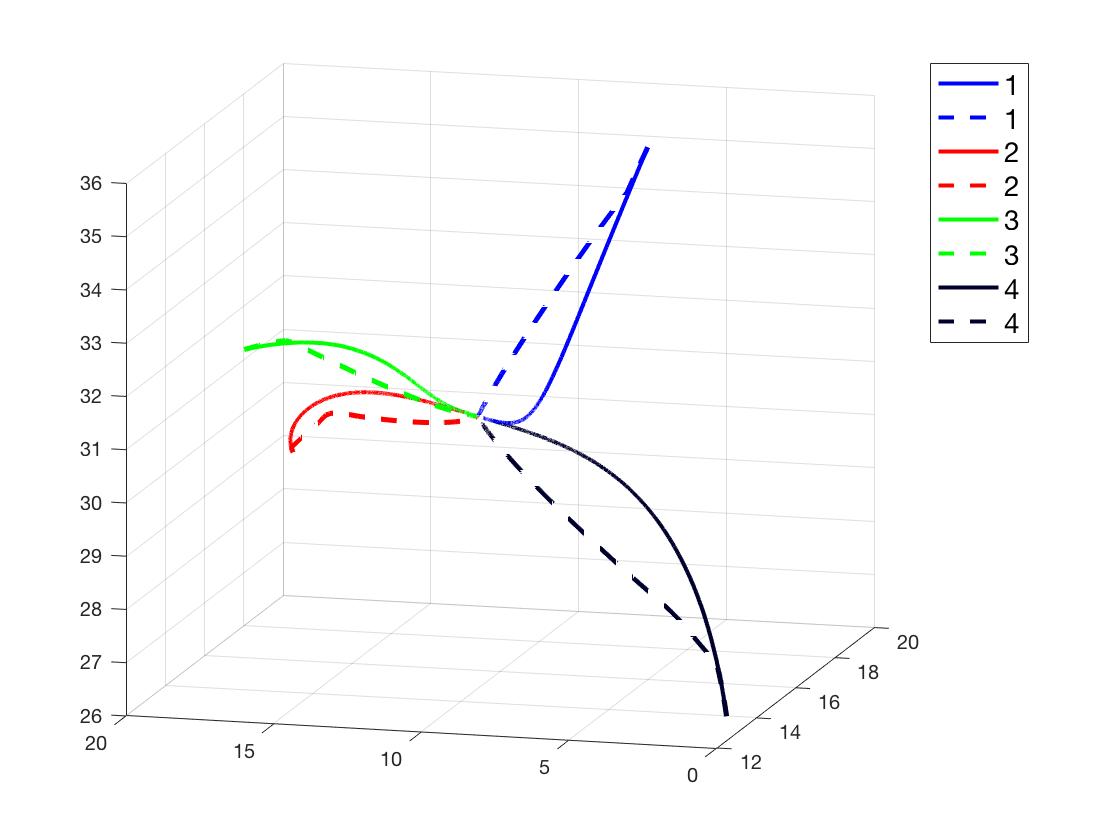}
\caption{Rendezvous missions comparing and unweighted network (solid curves) to a time-varying weighted one (dashed curves). They both agree on the consensus joint value. The points where the dashed curves diverge from the solid ones correspond to the addition of edges.}
\label{Fig-rendezvous3}
\end{figure}
In Figure \ref{Fig-rendezvous3-xyz} we compare the trajectories for agent 1 for each coordinate of motion. The eigenvalues for the unweighted network are given by $\{0,1,3,4\}$ and the eigenvalues for the weighted network are initially given by $\{0,6.234,19.385,37.655\}$, and near the rendezvous agreement are given by $\{0,0.190,0.196,0.207\}$.

\begin{figure}[h!]
\includegraphics[width=.325\linewidth]{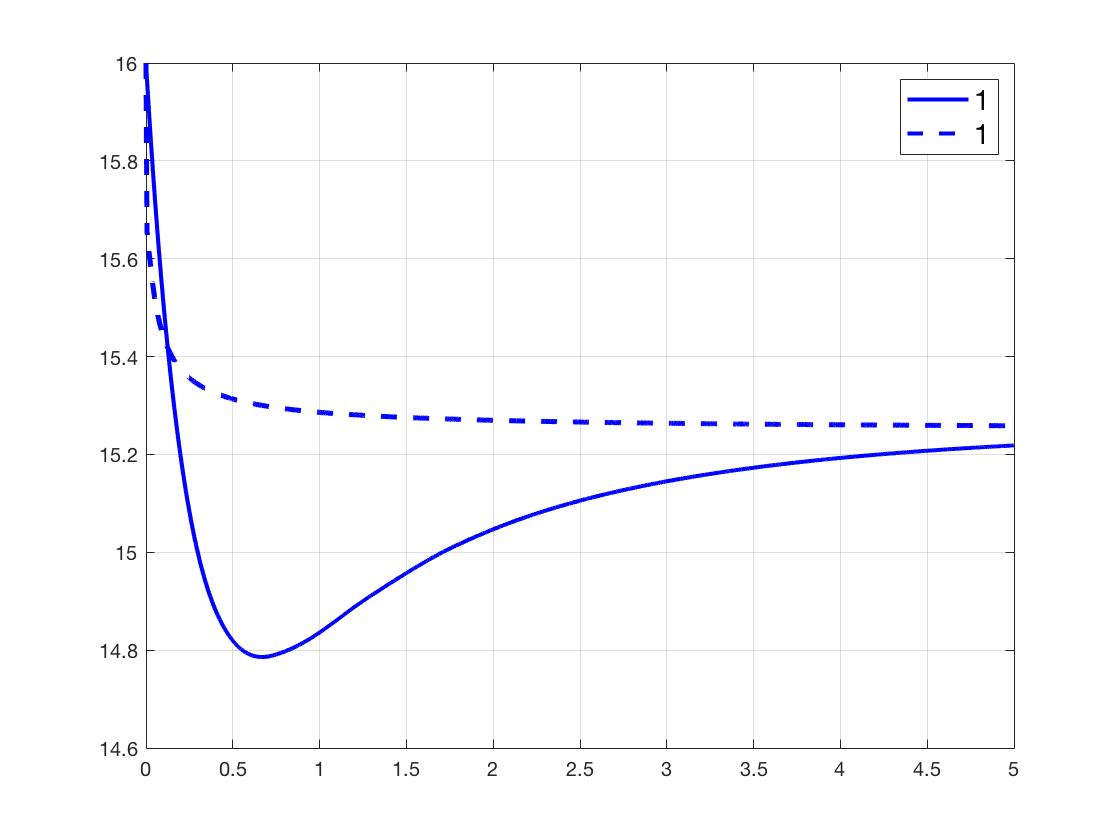}
\includegraphics[width=.325\linewidth]{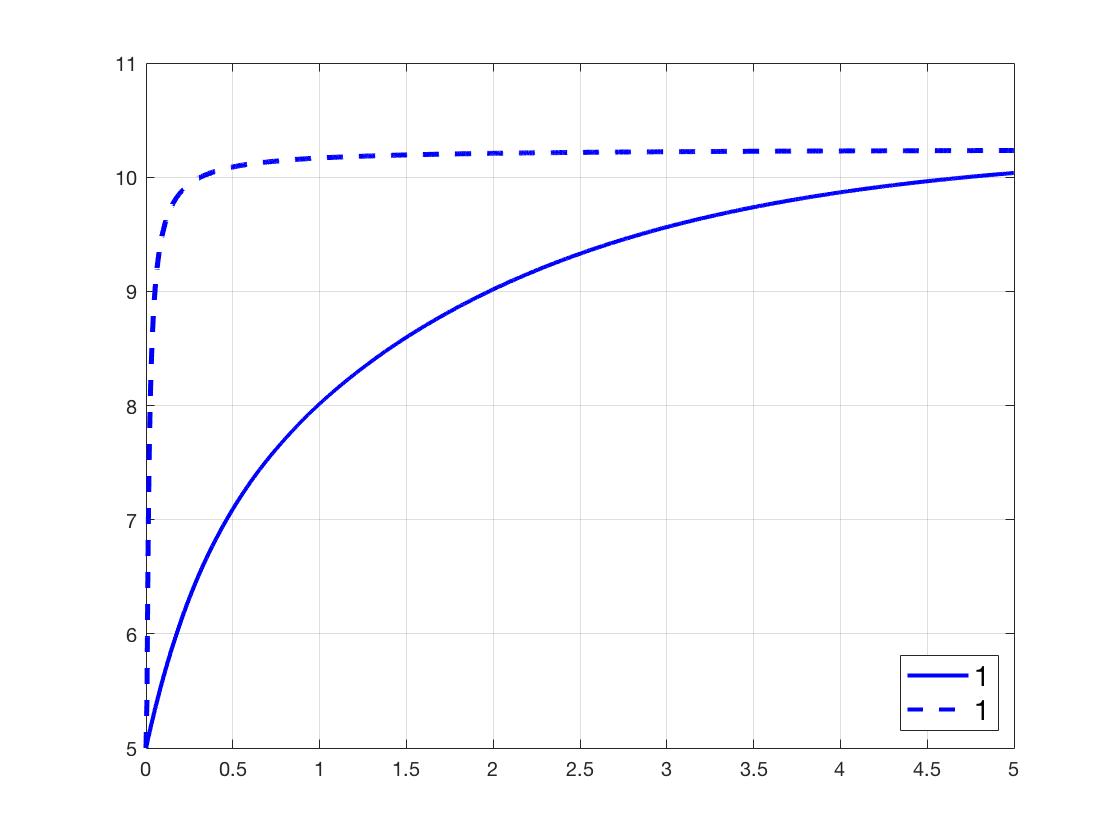}
\includegraphics[width=.325\linewidth]{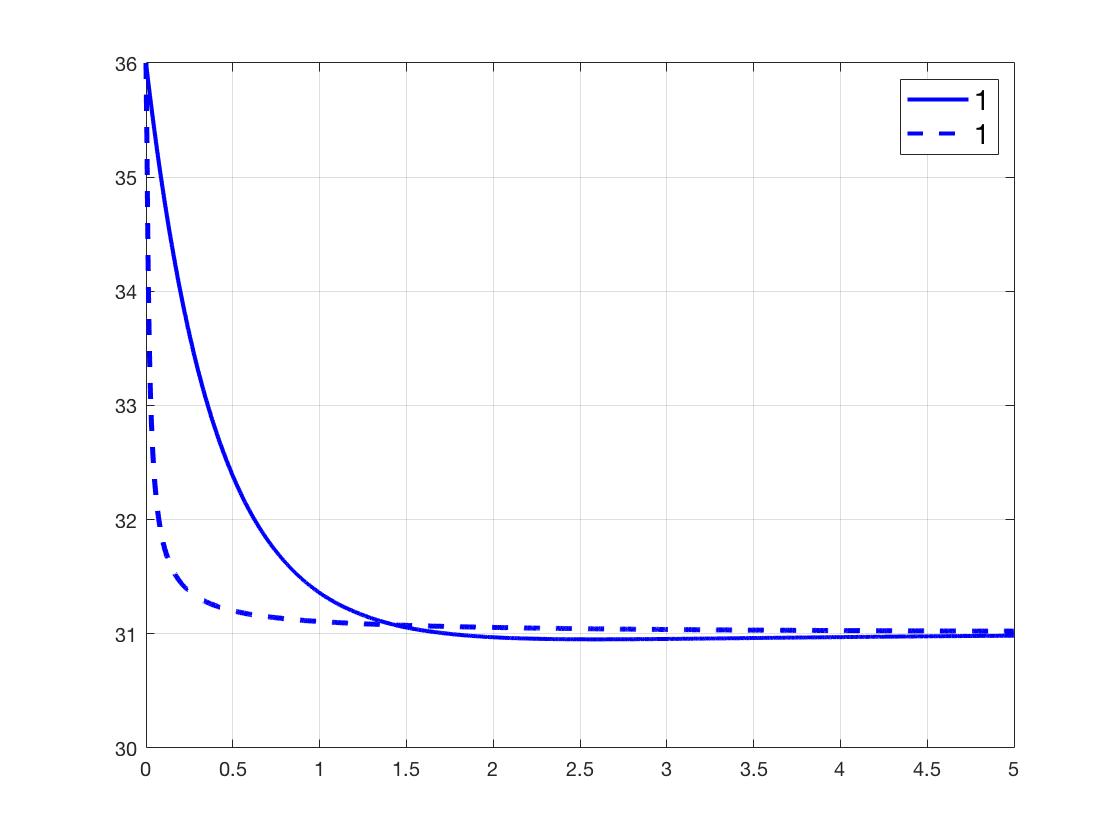}
\caption{Comparison of the $x,y,z$-motions for agent 1 for the rendezvous mission of Figure \ref{Fig-rendezvous3}. The components of the trajectory corresponding to the time-varying network clearly converge more rapidly.}
\label{Fig-rendezvous3-xyz}
\end{figure}

\newpage
\section{Dynamics of quadcopters}
\label{section-dynamics}

In order to successfully implement multi-agent motion planning, we need dynamically sound equations of motion governing the movement of each drone.  To this end, in this section we will derive such equations under certain realistic symmetry assumptions regarding the shape of the drone body.  Utilizing the geometric formalism of \cite{Bullo Lewis}, we present these equations in both first-order and second-order forms, and provide a differential geometric version in terms of the Levi-Civita connection on our configuration space equipped with a natural Riemannian metric.   Finally, we give the equations explicitly in coordinates.  Note that the notation in this section is independent from that in the previous section.

\subsection{Setup}
In this subsection we provide the basic physical notation and standard derivation of the equations of motion for a rigid body in the body-fixed frame.
A rigid body can model a number of vehicles in space or in a fluid, such as satellites or underwater vehicles, which can be stabilized by spinning rotors (\cite{Bloch1, Bloch3}).
Here, we assume our drone is nothing more than a rigid body in a fluid; later we will make certain additional assumptions explicit.  
We closely follow the derivations in \cite{DCDS, IJOC}.

The configuration space for a rigid body moving in a fluid (\cite{Bloch book}) is
$$Q=SE(3)\cong \mathbb R^3 \times SO(3).$$
Here we equip $Q$ with coordinates $(b,R)$ where $b\in \mathbb R^3$ represents the position of the center of mass of the body in space and $R\in \mathbb R^{3 \times 3}$ is a rotation matrix representing the orientation of the body aligned with the principal axes of inertia.  This coordinate system corresponds to some inertial frame of reference.

We regularly use the notation of the \textit{hat map}, the Lie algebra isomorphism
$$ \hat{} : (\mathbb R^3, \times) \to (\mathfrak{so}(3), [,])$$
given by $\hat{y}z = y \times z$, equivalently $\hat{y}=\begin{pmatrix}
0 & -y_3 & y_2 \\ y_3 & 0 & -y_1 \\ -y_2 & y_1 & 0
\end{pmatrix}$.

First, choose a non-inertial body-fixed frame 
and set the following notation:
\begin{align*}
v &= \text{translational velocity in body frame}\\
\Omega &= \text{angular velocity in body frame}\\
p &= \text{translational momentum in inertial frame}\\
\pi &= \text{angular momentum in inertial frame}\\
P &= \text{translational momentum in body frame}\\
\Pi &= \text{angular momentum in body frame}.
\end{align*}

We seek to derive the dynamic equations of motion for the system in the body frame.  Here this will take the form of a system of ODEs in the variables $(b, R, v, \Omega) \in TQ$.  Our kinematic equations are:
\begin{align}
\dot{b} &= Rv \label{b dot} \\
\dot{R} &= R\hat{\Omega} \label{R dot}.
\end{align}

We transform momenta between the inertial and body frames according to:
\begin{align}
p&=RP \label{transl mom} \\
\pi &= R\Pi + \hat{b} p \label{ang mom}
\end{align}

Begin by differentiating equations (\ref{transl mom}) and (\ref{ang mom}):
\begin{align}
\dot{p} &= \dot{R}P + R\dot{P} \label{little p dot} \\ 
\dot{\pi} &= \dot{R}\Pi + R \dot{\Pi} + \dot{b} \times p + b \times \dot{p}. \label{little pi dot}
\end{align}
To rewrite these equations in a single frame, we must introduce additional quantities.  In the inertial frame, let 
\begin{align*}
f_i &= \text{external forces}, \quad  i=1, \dots, k  \\
\tau_i &= \text{external torques}, \quad  i=1, \dots, l, 
\end{align*}
so that 
\begin{align}
E_F &= \sum_{i=1}^k R^tf_i \label{ext forces}= \text{total external force in body frame} \\
E_T &= \sum_{i=1}^l R^t\tau_i \label{ext torques}= \text{total external torque in body frame}.
\end{align}
Now in the inertial frame, the dynamics are solely due to external forces and torques:
\begin{align}
\dot{p} &= \sum_{i=1}^k f_i \label{p dot}\\
\dot{\pi} &= \sum_{i=1}^k \hat{x}_i f_i + \sum_{i=1}^l \tau_i, \label{pi dot}
\end{align}
where $x_i$ is the vector from inertial frame origin to line of action of $f_i$.
 
We then solve (\ref{little p dot}), (\ref{little pi dot}) for $\dot{P}$ and $\dot{\Pi}$ and simplify using equations (\ref{b dot}), (\ref{R dot}), (\ref{transl mom}), (\ref{ext forces}), (\ref{ext torques}), (\ref{p dot}), (\ref{pi dot}):  
\begin{align}
\dot{P} &= \hat{P}\Omega + E_F \label{P dot} \\
\dot{\Pi} &= \hat{\Pi}\Omega + \hat{P}v + E_T + R^t\sum_i(x_i-b) \times f_i. \label{Pi dot}
\end{align}
This gives the evolution of the momenta in the body frame, but the equations mix momenta and velocities.  In order to arrive at equations in terms of solely momenta or velocities, we must find explicit formulas relating the two.  This can be acheived by the Legendre transform
\begin{equation*}
P = \frac{\partial T}{\partial v}, \qquad
\Pi = \frac{\partial T}{\partial \Omega}.
\end{equation*}
To make this explicit, we introduce the kinetic energies
\begin{align}
T_{body} &= \frac{1}{2}\begin{pmatrix}
v \\ \Omega
\end{pmatrix}^t
\begin{pmatrix}
mI_3 & -m\hat{r}_{C_G} \\ m\hat{r}_{C_G} & J_b \end{pmatrix} 
\begin{pmatrix} v \\ \Omega \end{pmatrix} \label{kinetic body}  \\
T_{fluid} &= \frac{1}{2}\begin{pmatrix}
v \\ \Omega
\end{pmatrix}^t
\begin{pmatrix}
M_f & C_f^t \\ C_f & J_f \end{pmatrix} 
\begin{pmatrix} v \\ \Omega \end{pmatrix}  \label{kinetic fluid}\\
T &= T_{body} + T_{fluid} \nonumber \\
&= \frac{1}{2}\begin{pmatrix}
v \\ \Omega
\end{pmatrix}^t
\begin{pmatrix}
mI_3+M_f & -m\hat{r}_{C_G} +C_f^t\\ m\hat{r}_{C_G} +C_f & J_b +J_f \end{pmatrix} 
\begin{pmatrix} v \\ \Omega \end{pmatrix} \label{kinetic total} 
\end{align}
where
$m$ is the mass of the body, 
$r_{C_G}$ is vector from center of gravity to body frame origin, 
$J_b$ is the body inertia tensor, 
$J_f$ is the added mass inertia tensor,
$M_f$ is the added mass, 
$C_f$ is the added cross terms, and
$I_3$ is the $3 \times 3$ identity matrix.

Then by differentiating the total kinetic energy (\ref{kinetic total}), we find
\begin{align}
P=\frac{\partial T}{\partial v} &= (mI_3+M_f)v +(-m\hat{r}_{C_G} +C_f^t)\Omega  \label{transform transl}\\
\Pi=\frac{\partial T}{\partial \Omega} &= ( m\hat{r}_{C_G} +C_f)v + (J_b +J_f)\Omega.\label{transform ang}
\end{align}
These equations transform velocities into momenta.

We now assume the body has an especially nice shape, which is realistic for most drones.

\textbf{Assumption 1.}  Assume the body has three planes of symmetry and the principle axes of inertial coincide with the body frame axes.  This implies $J_b, J_f,$ and $M_f$ are all diagonal and $C_f =0$.
When we consider gravity in Section \ref{subsec: extrenal}, we also assume that this symmetry extends to the mass density of the body.

\textbf{Assumption 2.}  Assume the center of gravity of the body coincides with the origin of the body frame.  This implies $\hat{r}_{C_G} =0$.

With these two assumptions, equations (\ref{transform transl}) and (\ref{transform ang}) reduce to
\begin{align}
P &= Mv  \label{transform transl reduced}  \\
\Pi &= J\Omega,  \label{transform ang reduced} 
\end{align}
where 
$M=mI_3 +M_f$ is the total mass of the system and 
$J= J_b +J_f$ is the moment of inertia of the system.

Now differentiate (\ref{transform transl reduced}) and (\ref{transform ang reduced}) and simplify using (\ref{P dot}), (\ref{Pi dot}), (\ref{transform transl reduced}) and (\ref{transform ang reduced}) to obtain the derivatives of the velocities expressed solely in terms of positions and velocities, without momenta:
\begin{align}
m\dot{v} &= Mv \times \Omega + E_F \label{eqn of motion transl} \\
J\dot{\Omega} &= J\Omega \times \Omega + Mv \times v +  E_T + R^t\sum_i(x_i-b) \times f_i.\label{eqn of motion ang}
\end{align}

Finally, since for drones our fluid is air, we make the following simplifying assumption.

\textbf{Assumption 3.}  Assume the added fluid is negligible.  This implies $M_f=J_f=0$. Then $J=J_b$, and since $M=mI_3$ is a scalar matrix, the fictitious force $Mv \times v=0$.

Combining this with (\ref{b dot}), (\ref{R dot}), (\ref{eqn of motion transl}), (\ref{eqn of motion ang}), we obtain the complete equations of motion for a rigid drone in air as a first order system on $TQ$ coordinatized by $(b, R, v, \Omega)$:
\begin{align}
\dot{b} &= Rv \\
\dot{R} &= R\hat{\Omega} \\
m\dot{v} &= mv \times \Omega + E_F  \\
J\dot{\Omega} &= J\Omega \times \Omega + 
E_T + R^t\sum_{i=1}^k(x_i-b) \times f_i. \label{Omega dot}
\end{align}

\subsection{External forces and torques}\label{subsec: extrenal}
In this section we explicitly work out the external terms $E_F$ and $E_T$ for a drone in the body frame.  We have four rotors producing thrust in a quadcopter, with the $i$th rotor possessing  angular velocity  $\omega_i$ producing thrust $t_i=K_r\omega_i^2$, where $K_r$ is the thrust coefficient (\cite{Mueller}).  Here we are implicitly treating the $\omega_i$ as our controls.

For the external forces, we follow \cite{Stepanyan}.  
The force due to drag is 
$$f_1=-\text{diag}(v_1|v_1|,v_2|v_2|,v_3|v_3|)C_D,$$ 
where $v_j$ is the $j$th component of the linear velocity $v$ and $C_D$ is the vector of translational drag coefficients in the body frame.
The force due to gravity is 
$$f_2=-mgR^te_3^I,$$
where $g$ is the gravitational constant and  the third inertial frame basis vector $e_3^I$ represents the opposite direction of gravitational attraction.
Lastly, the force produced by thrust is $$f_3=\sum_{i=1}^4e_3^Bt_i,$$
where the third body frame basis vector $e_3^B$ represents the direction in which each of the motors produce thrust.
Thus $E_F=f_1+f_2+f_3$.  

Ostensibly, each of these forces induces a torque, represented by the last term $R^t\sum_{i=1}^k(x_i-b) \times f_i$ in (\ref{Omega dot}).  However, for $f_1$ and $f_2$, no torque is induced as the line of force passes through the center of mass of the body due to our symmetry assumptions.  
The force $f_3$, however, represents the sum of four individual forces located at each of the four rotors.  Each of these forces induces a torque, and their sum gives the net torque generated by the rotors:
$$\tau_f=\begin{pmatrix}K_rd(\omega_3^2-\omega_1^2)\\K_rd(\omega_4^2-\omega_2^2)\\K_d\sum_{i=1}^4(-1)^{i+1}\omega_i^2
    \end{pmatrix},$$
where 
$K_d$ is the propeller drag coefficient and $d$ is the distance from the drone center of mass to the rotation axis of each rotor.  

The external torques are computed as in \cite{Bouadi, Stepanyan}.
The torque in the body frame generated by drag is 
$$\tau_1=-\text{diag}(\Omega_1|\Omega_1|,\Omega_2|\Omega_2|,\Omega_3|\Omega_3|)C_\tau,$$ 
where $\Omega_j$ is the $j$th component of the angular velocity $\Omega$ and $C_{\tau}$ is the vector of rotational drag coefficients in the body frame.
The torque due to gyroscopic effects is 
$$\tau_2=\sum_{i=1}^4\Omega\times (-1)^{i+1}J_r(0,0,\omega_i)^t,$$
where $J_r$ is the moment of inertia for a rotor.  Assuming the rotor has nontrivial inertia about only the vertical axis, $J_r$ has only one nonzero entry which we denote by $\bar{J_r}$, so that $J_r = \text{diag}(0, 0, \bar{J_r})$.
We then have $E_T=\tau_1  + \tau_2$.

We can now give our updated equations of motion for a quadcopter:
\begin{align}
\dot{b} &= Rv \label{b dot again} \\ 
\dot{R} &= R\hat{\Omega}  \label{R dot again} \\
m\dot{v} &= mv \times \Omega + f_1 +f_2+f_3  \label{v dot again} \\
J\dot{\Omega} &= J\Omega \times \Omega + \tau_f + \tau_1 + \tau_2. \label{Omega dot again}
\end{align}

\subsection{Coordinate expression}
\label{subsec: coord expr}
Here we work in coordinates $\eta=(b_1, b_2, b_3, \phi, \theta, \psi)$ on $Q=SE(3)$ representing the position and orientation of the drone in the inertial frame.  The positions $b_1, b_2, b_3$ are the standard coordinates for $\mathbb R^3$.
The angles $ \phi, \theta, \psi$ are Tait-Bryan angles, known in aeronautics as \textit{roll, pitch}, and \textit{yaw} respectively, and sometimes referred to as \textit{Euler angles} (although not the proper or classical kind).  The rotation angles are given intrinsically by the sequence $z-y'-x''$  or extrinsically by the sequence $x-y-z$.  
Here we take $(\phi, \theta, \psi) \in (-\pi, \pi)\times (-\pi/2, \pi/2) \times (-\pi, \pi) .$  These coordinates on $Q$ induce natural coordinates for the velocities $(v,\Omega)=(v_1, v_2, v_3,\Omega_1, \Omega_2, \Omega_3)$.

Now let
$$
R(\eta) = \begin{pmatrix}
        C_\psi C_\theta & C_\psi S_\theta S_\phi - S_\psi C_\phi & C_\psi S_\theta C_\phi + S_\psi S_\phi\\
        S_\psi C_\theta & S_\psi S_\theta S_\phi + C_\psi C_\phi & S_\psi S_\theta C_\phi - C_\psi S_\phi\\
        -S_\theta & C_\theta S_\phi & C_\theta C_\phi
    \end{pmatrix},
    $$
and
    $$\Theta(\eta)=\begin{pmatrix}1&S_\phi T_\theta&C_\phi T_\theta\\0&C_\phi&-S_\phi\\0&\frac{S_\phi}{C_\theta}&\frac{C_\phi}{C_\theta}\end{pmatrix},$$
    where $C_x=\cos(x),$ $S_x=\sin(x)$ and $T_x=\tan(x)$.
We can write $R$ as the composition of three pure rotations as follows.  Let $R_{i, \alpha}$ denote rotation by $\alpha$ radians about the $i$th principal axis of inertia.
Explicitly,  
\begin{align*}
&R_{1, \phi} = 
\begin{pmatrix}
        1 & 0 & 0 \\
        0 &  C_\phi & -S_\phi\\
        0 & S_\phi & C_\phi 
    \end{pmatrix},
\\
    &R_{2, \theta} = 
\begin{pmatrix}
        C_\theta & 0 & S_\theta \\
        0 & 1 & 0 \\
        -S_\theta & 0 & C_\theta
    \end{pmatrix},
\\
    &R_{3, \psi} = 
\begin{pmatrix}
        C_\psi & -S_\psi & 0 \\
        S_\psi & C_\psi & 0 \\
        0 & 0 & 1
    \end{pmatrix}.
\end{align*}
Then $R=   R_{3, \psi}    R_{2, \theta}  R_{1, \phi} $.

Now $R$ transforms linear velocity in the body frame to linear velocity in the inertial frame, and
$\Theta$ transforms the body-fixed angular velocity $(\Omega_1, \Omega_2, \Omega_3)$ into the Euler rate vector $(\dot{\phi}, \dot{\theta}, \dot{\psi})$.
Thus our kinematic equations (\ref{b dot again}) and (\ref{R dot again}) take the form
 $$\dot{\eta}=\begin{pmatrix}R(\eta)&0\\0&\Theta(\eta)\end{pmatrix}\begin{pmatrix}v\\\Omega\end{pmatrix}.$$
Substituting $R$ into equations (\ref{v dot again}) and (\ref{Omega dot again}) and expanding yields the following.
 
 \begin{lemma} \label{lemma: eqns in coords}
 In coordinates developed above, the equations of motion for a quadcopter take the form
 \begin{align}
 \dot{b}_1&=v_1C_\psi C_\theta+v_2(C_\psi S_\theta S_\phi - S_\psi C_\phi) +v_3(C_\psi S_\theta C_\phi + S_\psi S_\phi)\\
        \dot{b}_2&=v_1S_\psi C_\theta+v_2(S_\psi S_\theta S_\phi + C_\psi C_\phi)+v_3(S_\psi S_\theta C_\phi - C_\psi S_\phi)\\
        \dot{b}_3&=-v_1S_\theta +v_2C_\theta S_\phi +v_3C_\theta C_\phi\\
        \dot{\phi}&=\Omega_1+\Omega_2S_\phi T_\theta +\Omega_3C_\phi T_\theta\\
        \dot{\theta}&=\Omega_2C_\phi-\Omega_3S_\phi\\
        \dot{\psi}&=\Omega_2\frac{S_\phi}{C_\theta}+\Omega_3\frac{C_\phi}{C_\theta} \\        
 \dot{v}_1&=v_2\Omega_3-v_3\Omega_2-\frac{1}{m}v_1|v_1|C_{D_1}+gS_\theta\\
        \dot{v}_2&=v_3\Omega_1-v_1\Omega_3-\frac{1}{m}v_2|v_2|C_{D_2}-gC_\theta S_\phi\\
        \dot{v}_3&=v_1\Omega_2-v_2\Omega_1+\frac{1}{m}\Bigg(\sum_{i=1}^4 t_i-v_3|v_3|C_{D_3}\Bigg)-gC_\theta C_\phi \\        
  \dot{\Omega}_1&=\frac{1}{J_{1}}\left[(J_{2}-J_{3})\Omega_2\Omega_3+\bar J_r\sum_{i=1}^4(-1)^{i+1}\Omega_2\omega_i
  +K_rd(\omega_3^2-\omega_1^2)-\Omega_1|\Omega_1|C_{\tau_1}\right]\\
        \dot{\Omega}_2&=\frac{1}{J_{2}}\left[(J_{3}-J_{1})\Omega_1\Omega_3-\bar J_r\sum_{i=1}^4(-1)^{i+1}\Omega_1\omega_i
        +K_rd(\omega_4^2-\omega_2^2)-\Omega_2|\Omega_2|C_{\tau_2}\right]\\
 \dot{\Omega}_3&=\frac{1}{J_{3}}\left[(J_{1}-J_{2})\Omega_1\Omega_2+K_d\sum_{i=1}^4(-1)^{i+1}\omega_i^2-\Omega_3|\Omega_3|C_{\tau_3}\right].        
 \end{align}
 \end{lemma}

\subsection{Geometric formulation}
Here we will equip $Q$ with a natural metric allowing us to characterize the equations of motion as a second order system on $Q$.  We then lift the dynamics to a first order system on $TQ$.  

Under our assumptions, the kinetic energy (\ref{kinetic total}) can be encoded as the matrix
$$ \mathbb G = 
\begin{pmatrix}
mI_3 & 0 \\ 0 & J
\end{pmatrix}$$
where $m$ and $J$ denote the mass and inertia tensor of the drone, respectively.
This is a (diagonal) positive definite symmetric matrix and induces a Riemannian metric on $Q$, which in turn yields the Levi-Civita connection $\nabla$.  This allows us to define the \textit{acceleration} of a curve in $Q$ as $\nabla_{\dot{\gamma}} \dot{\gamma}$, where a curve $\gamma(t) = (b(t), R(t))$ has velocity $\dot{\gamma} = (b(t), R(t), v(t), \Omega(t))$.  Explicitly, we compute
$$\nabla_{\dot{\gamma}} \dot{\gamma} = 
\begin{pmatrix}
\dot{v} + \Omega \times v \\ \dot \Omega + J^{-1}(\Omega \times J\Omega)
\end{pmatrix}.
$$
This connection evidently contains the fictitious forces, and is independent of choice of coordinates.  The geodesics for $\nabla$ (curves with zero acceleration) represent motions of a drone subject to no external forces.

Now in general, a Riemannian metric $\mathbb G$ on a manifold $Q$ induces the ``musical isomorphisms"  $\mathbb G^{\sharp}: T^*Q \to TQ$ and $\mathbb G^{\flat}: TQ \to T^*Q$.
Newton's equation can be expressed geometrically as $F(\gamma) = \mathbb G^{\flat} \nabla_{\dot{\gamma}} \dot{\gamma}$ where $F$ is some force.  In our setting, this produces the geometric version of our equations of motion:
\begin{equation}
 \nabla_{\dot{\gamma}} \dot{\gamma} =   \mathbb G^{\sharp}F(\gamma) \label{geometric eqns}
\end{equation}
where
$$ \mathbb G^{\sharp} = 
\begin{pmatrix}
m^{-1}I_3 & 0 \\ 0 & J^{-1}
\end{pmatrix}$$
and 
$$ F(b, R) = 
\begin{pmatrix}
f_1 +f_2 +f_3  \\  \tau_f + \tau_1 + \tau_2
\end{pmatrix}.$$

Equation (\ref{geometric eqns}) represents a second order system on $Q$.  We now express this as a first order system on $TQ$.  To that end, we invoke the \textit{vertical lift} $\text{vlft}_w: T_qQ \to T_wTQ$ where $w \in T_qQ$,  and the \textit{geodesic spray} $S:TQ \to TTQ$ (\cite{LewisMurray, DCDS}).
For our system, we can explicitly compute
$$S(b,R,v,\Omega)=\begin{pmatrix}v\\\Omega\\v\times\Omega\\J^{-1}(J\Omega\times\Omega)\end{pmatrix}.$$
Here the components are expressed relative to the standard left-invariant basis of vector fields $\{ Y_k\}$ on $TQ=TSE(3)$, where at a point $(b, R)\in SE(3)$ we have $Y_i=(Re_i,0)$ for $i=1, 2, 3$ and $Y_{j+3}=(0, R\hat{e}_j)$ for $j=1, 2, 3$.

Then our equations of motion can be expressed by the first order system
\begin{equation}\label{Gamma dot}
\dot \Gamma = X(\Gamma)
\end{equation}
where $\Gamma$ is a curve in $TQ$ and $X$ is the vector field on $TQ$ given by 
\begin{equation}
X = S + \text{vlft} (\mathbb G^\sharp F).
\end{equation}

Next, we express this system as an affine nonlinear control system as in \cite{Bloch book}.  
We can rewrite equation (\ref{Gamma dot}) as
\begin{equation}\label{Gamma dot again}
    \dot\Gamma = 
    \begin{pmatrix}
    v\\\Omega\\v\times\Omega +m^{-1}f_1+m^{-1}f_2\\J^{-1}(J\Omega\times\Omega)+J^{-1}\tau_1
    \end{pmatrix}+
    \begin{pmatrix}
    0\\0\\ m^{-1}f_3\\J^{-1}\tau_f +J^{-1}\tau_2    \end{pmatrix},
\end{equation}
where the first summand is the drift, and the second summand contains all the controls $\omega_i$.
Explicitly, we have
\begin{align*}
    m^{-1}f_3&= m^{-1}\sum_{i=1}^4K_r\omega^2_i e_3^B \\
    J^{-1}\tau_f +J^{-1}\tau_2&= J^{-1}\begin{pmatrix}K_rd(\omega_3^2-\omega_1^2)\\K_rd(\omega_4^2-\omega_2^2)\\K_d\sum_{i=1}^4(-1)^{i+1}\omega_i^2 \end{pmatrix}
    +J^{-1}\sum_{i=1}^4\Omega\times (-1)^{i+1}J_r
    \begin{pmatrix}
    0 \\ 0 \\ \omega_i
    \end{pmatrix}.
\end{align*}

Note that the torque due to gyroscopic effects is linear in $\omega_i$, while the force and torque due to thrust are linear in $\omega_i^2$.  Consequently we cannot write this system as an affine control system in independent controls.  As is common in the literature (e.g. \cite{Mueller, Stepanyan}) we now simplify by ignoring the gyroscopic effects in $\tau_2$.  This simplification does not affect the simple motion planning in Section \ref{subsec: planning}, but the gyroscopic effects are present in the simulations of Section \ref{subsec: simulations}.
Consider the controls
$$u_i=\omega_i^2 \quad \text{for} \quad i=1,2,3,4. 
$$

\begin{proposition}\label{prop: control system}
Ignoring gyroscopic effects, the quadcopter can be expressed as an affine nonlinear control system as:
\begin{equation}
\dot{\Gamma}=f(\Gamma) + \sum_{i=1}^4g_i(\Gamma)u_i
\end{equation}
where the drift vector field is 
$$f(\Gamma)= \begin{pmatrix}
    v\\\Omega\\v\times\Omega +m^{-1}f_1+m^{-1}f_2\\J^{-1}(J\Omega\times\Omega)+J^{-1}\tau_1
    \end{pmatrix}
$$
and the control vector fields are 
\begin{eqnarray*}
    g_1=\begin{pmatrix}
    0\\ 0\\ (0, 0,m^{-1}K_r)^t \\ (-K_rd, 0, K_d)^t
    \end{pmatrix}, \quad
    g_2=\begin{pmatrix}
    0\\ 0\\ (0, 0,m^{-1}K_r)^t \\ (0, -K_rd, -K_d)^t
    \end{pmatrix}, \\
    g_3=\begin{pmatrix}
    0\\ 0\\ (0, 0,m^{-1}K_r)^t \\ (K_rd, 0, K_d)^t
    \end{pmatrix}, \quad
    g_4=\begin{pmatrix}
    0\\ 0\\ (0, 0,m^{-1}K_r)^t \\ (0, K_rd, -K_d)^t
    \end{pmatrix}.
\end{eqnarray*}
\end{proposition}

\section{Quadcopter multi-agent dynamics}
\label{section-quadtrajectories}

\subsection{Single-agent motion planning}
\label{subsec: planning}
We first explore some simple motions for a single quadcopter.  We use the coordinates and notation from Section \ref{section-dynamics}.

 \begin{lemma}\label{lemma: yaw motion}
Suppose a quadcopter undergoes yaw only, with no translational or other rotational motion.  Then
\begin{equation*}
    \omega_1=\omega_3=\sqrt{\frac{mg}{2K_r}-\omega_2^2} \qquad \text{and} \qquad \omega_4=\omega_2,
\end{equation*}
where $\omega_2$ is a free parameter.
\end{lemma}

\begin{proof}
The desired motion keeps all variables constant except $\Omega_3$ and $\psi$.  This imposes $\omega_1=\omega_3$ and $\omega_2=\omega_4$ as well as 
$$\dot{b}= v_1 = v_2 = \theta = \phi = \Omega_1 = \Omega_2=0
$$
for the duration of the motion. 

The equations of motion in Lemma \ref{lemma: eqns in coords} simplify considerably; in particular, we have $\dot{b}_3=v_3$ and $\dot{v}_3=\frac{1}{m}\sum_{i=1}^4 t_i-g$.  Assuming $\dot{b}_3=0$ gives the expected result for hovering: $\sum_{i=1}^4 t_i=mg$.  Using $\omega_1=\omega_3$ and $\omega_2=\omega_4$ with $t_i=K_r\omega_i^2$ allows us to solve for 
$\omega_1=\sqrt{\frac{mg}{2K_r}-\omega_2^2}.$

\end{proof}

\begin{lemma}\label{lemma: x-axis motion}
Suppose a quadcopter moves only in the direction of the first body axis.  Then 
\begin{equation*}
    \omega_2=\frac{1}{3}\left(\sqrt{\frac{6a-8K_r\omega_4^2}{K_r}}-\omega_4\right) \qquad \text{and} \qquad \omega_1=\omega_3=\frac{1}{2}(\omega_2+\omega_4),
\end{equation*}
where 
$$a=m\left(\frac{1}{C_\theta}(\dot{v}_1S_\theta+v_1C_\theta\Omega_2+v_3S_\theta\Omega_2)-v_1\Omega_2+gC_\theta\right)+v_3|v_3|C_{D_3}
$$
and $\omega_4$ is a free parameter.  
\end{lemma}

\begin{proof}
We desire to move only along the body $x$-axis with no motion in the body $y$ or $z$ direction, requiring the body to pitch without any yaw or roll.  This imposes $\omega_1=\omega_3$ as well as 
$$\dot{b}_3=\phi=V_2=\Omega_3=\Omega_1=0
$$
for the duration of the motion.

The equations of motion in Lemma \ref{lemma: eqns in coords} simplify considerably; in particular, we have $\dot{\theta}=\Omega_2$ and
$\dot{b}_3=-V_1S_\theta+V_3C_\theta.$
Using $\omega_1=\omega_3$ and setting $\dot{\Omega}_1=0$ forces $\omega_1=\omega_3=\frac{1}{2}(\omega_2+\omega_4)$.  Setting $\Ddot{b}_3=0$ and using $\dot{\theta}=\Omega_2$ gives
$$\dot{v}_3=\frac{1}{C_\theta}(\dot{v}_1S_\theta+v_1C_\theta\Omega_2+v_3S_\theta\Omega_2).
$$
Setting this equal to the equation of motion for $\dot{v}_3$, we find 
$$\sum_{i=1}^4 t_i = m\left(\frac{1}{C_\theta}(\dot{v}_1S_\theta+v_1C_\theta\Omega_2+v_3S_\theta\Omega_2)-v_1\Omega_2+gC_\theta\right)+v_3|v_3|C_{D_3}.
$$
Denote the right side of this equation by $a$.
Finally, using $t_i=K_r\omega_i^2$, we can solve for
$$
\omega_2=\frac{1}{3}\left(\sqrt{\frac{6a-8K_r\omega_4^2}{K_r}}-\omega_4\right).
$$
\end{proof}

Note that in Lemma \ref{lemma: x-axis motion} we could alternatively treat $\omega_4$ as the free parameter and solve for an identical expression for $\omega_4$ in terms of $\omega_2$.  We have similar choices for the yaw motion in Lemma \ref{lemma: yaw motion}.

Moreover, a nearly identical calculation gives path planning along the second body axis.  Treating $\omega_3$ as the free parameter, we can move only in the body $y$ direction using 
\begin{equation*}
    \omega_1=\frac{1}{3}\left(\sqrt{\frac{6a-8K_r\omega_3^2}{K_r}}-\omega_3\right) \qquad \text{and} \qquad \omega_2=\frac{1}{2}(\omega_1+\omega_3).
\end{equation*}
In order to move in another direction in the body $xy$-plane (with no motion in the $z$ direction), we can simply combine the motions in Lemmas \ref{lemma: yaw motion} and \ref{lemma: x-axis motion}.

\subsection{Simulations}\label{subsec: simulations}

\subsubsection{Single agent}

We demonstrate the simple motions from Section \ref{subsec: planning} in Figures \ref{fig: up and over config} and \ref{fig: up and over controls}.  In the coordinates of Section \ref{subsec: coord expr}, the drone begins at $(0,0,0,0,0,0)\in SE(3)$ with motors producing the necessary amount of thrust to remain at a hovering equilibrium.  Over the first 12 seconds, the drone maintains orientation while rising vertically in the body (and inertial) $z$-direction, ending at another hover.  Over the next four seconds, the drone maintains position while rotating $\pi/2$ radians counterclockwise about the body $z$-axis as described in Lemma \ref{lemma: yaw motion}.  Finally, for the last four seconds the drone executes the motion in Lemma \ref{lemma: x-axis motion}, flying directly along the body $x$-axis.  Figure \ref{fig: up and over config} shows the path of the drone in space as well as the orientation angles over time.  Figure \ref{fig: up and over controls} shows the controls used to produce this motion.

\begin{figure}
\centering
    \includegraphics[scale=0.08]{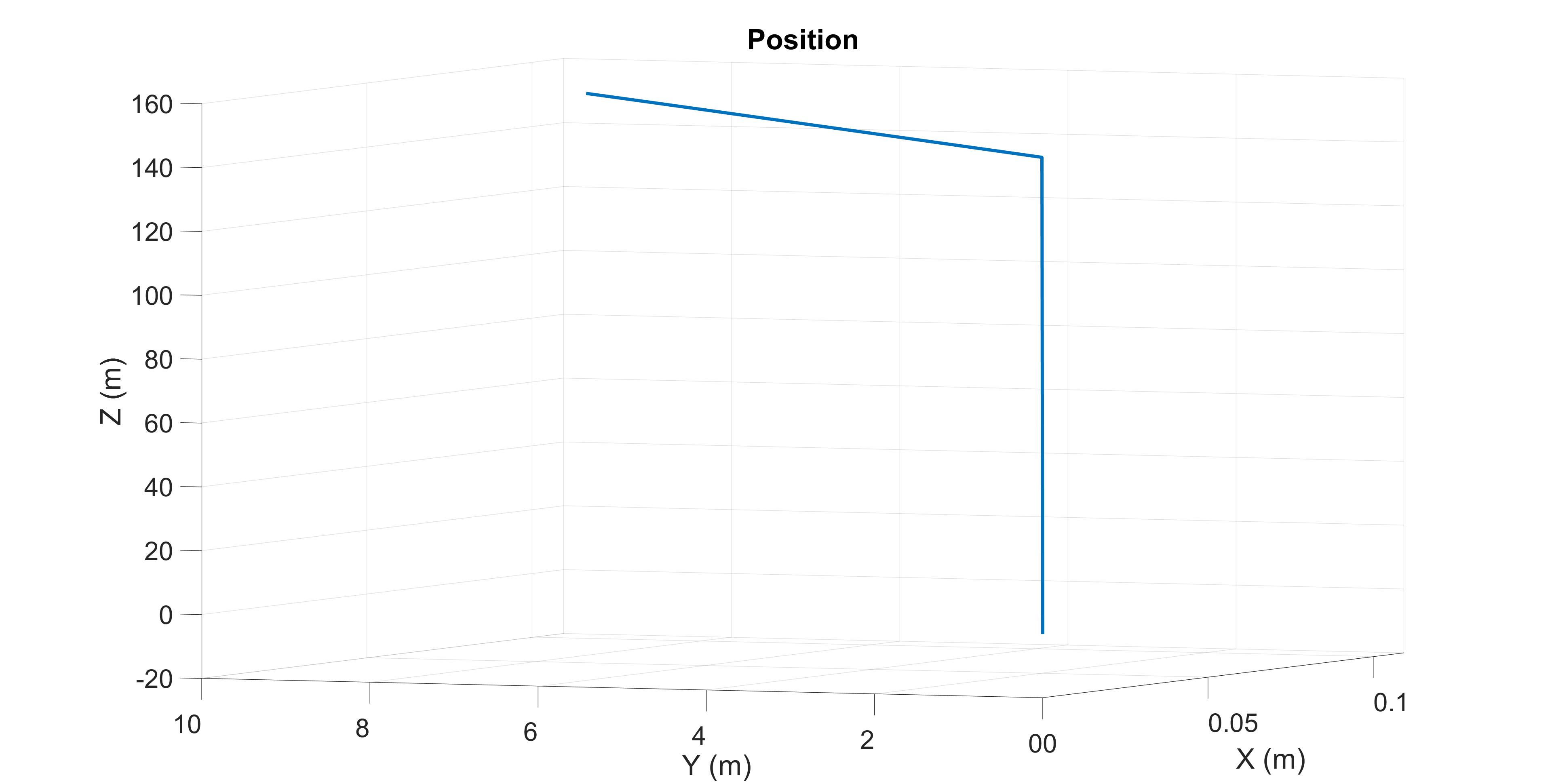}
    \includegraphics[scale=0.08]{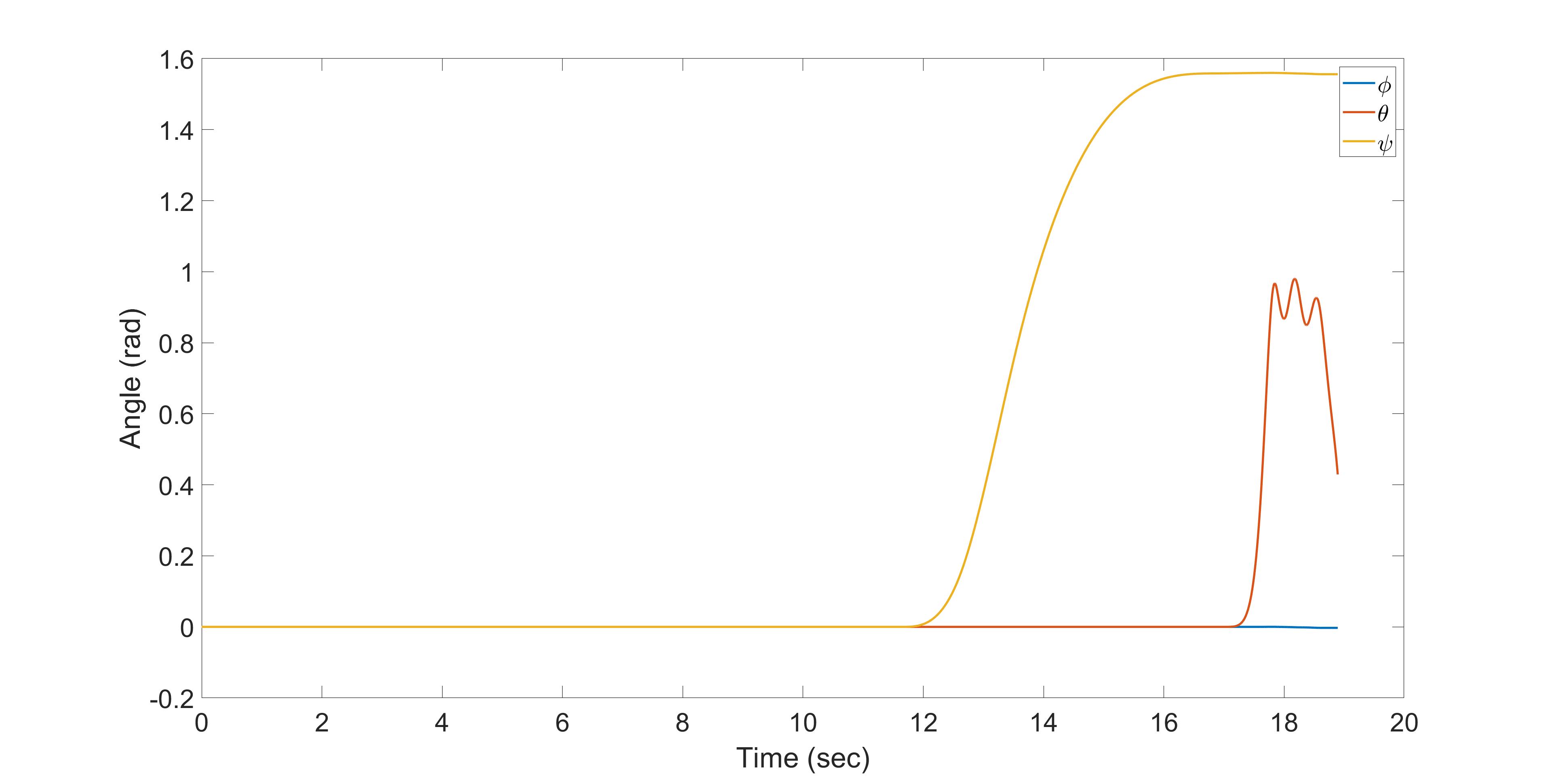}
    \caption{A quadcopter rising straight up, then yawing while hovering, then flying straight along the body $x$-axis.  Top: position in space.  Bottom: orientation angles over time.}\label{fig: up and over config}
\end{figure}

\begin{figure}
\centering
    \includegraphics[scale=0.08]{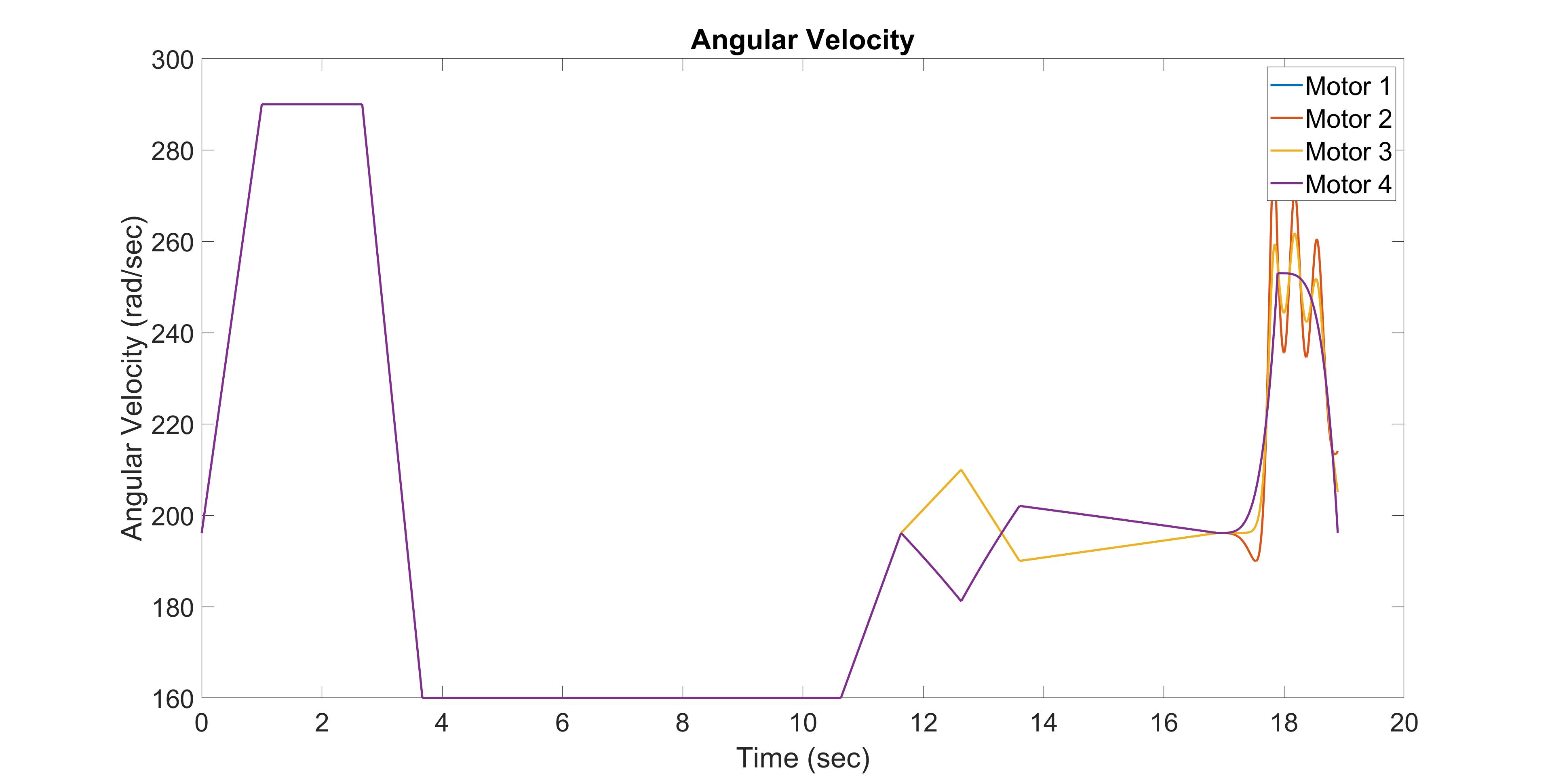}
    \includegraphics[scale=0.08]{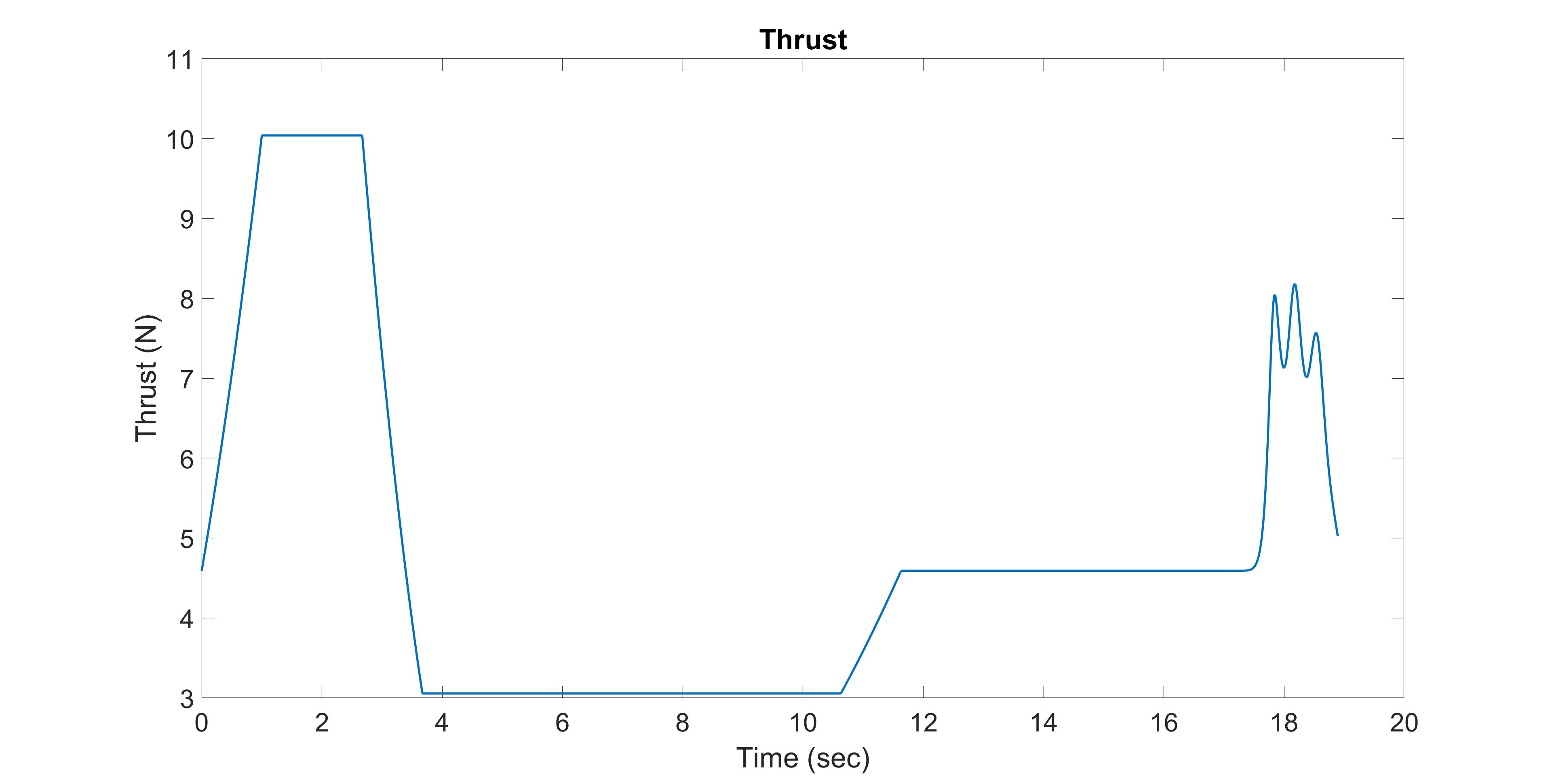}
    \caption{The controls used to produce the motion in Figure \ref{fig: up and over config}.  Top: angular velocities of each of the four motors over time.  Bottom: total thrust over time.}\label{fig: up and over controls}
\end{figure}

\subsubsection{Multi-agent}

Here we combine most of the work completed above in order to simulate an agreement protocol for dynamically sound quadcopters.  Consider three drones on a rendezvous mission with given initial conditions.  We assume each drone begins hovering with all translational and angular velocities zero.  This forces each initial pitch and roll to also be zero, but we choose different initial yaws.  For simplicity, we assume all three drones start in the $z=0$ plane (which need not correspond to the ground) but choose different initial $x,y$ positions.  This simplifying assumption does not sacrifice much generality: if the three drones started at three different nonzero altitudes, we would simply find the average initial altitude and execute simple motions bringing each drone up or down to this height.  

In the coordinates $(b_1, b_2, b_3, \phi, \theta, \psi) \in SE(3)$ from Section \ref{subsec: coord expr}, we choose our three drones to have initial configurations  
$(0,0,0, 0, 0, 0)$, $(0,9,0, 0, 0, -\pi/4)$, and $(15,9,0, 0, 0, \pi/2)$.

For this example we will assume a complete network $G=K_3$, so by Theorem \ref{thm: connected graph} we have that all trajectories are straight lines.  In particular, each drone will traverse the line segment connecting their initial position in $\mathbb R^3$ to the rendezvous position, which is simply the component-wise average of the initial conditions: $(5, 6, 0).$  Figure \ref{fig: rendezvous traj} shows these trajectories.

\begin{figure}
    \centering
    \includegraphics[scale=0.08]{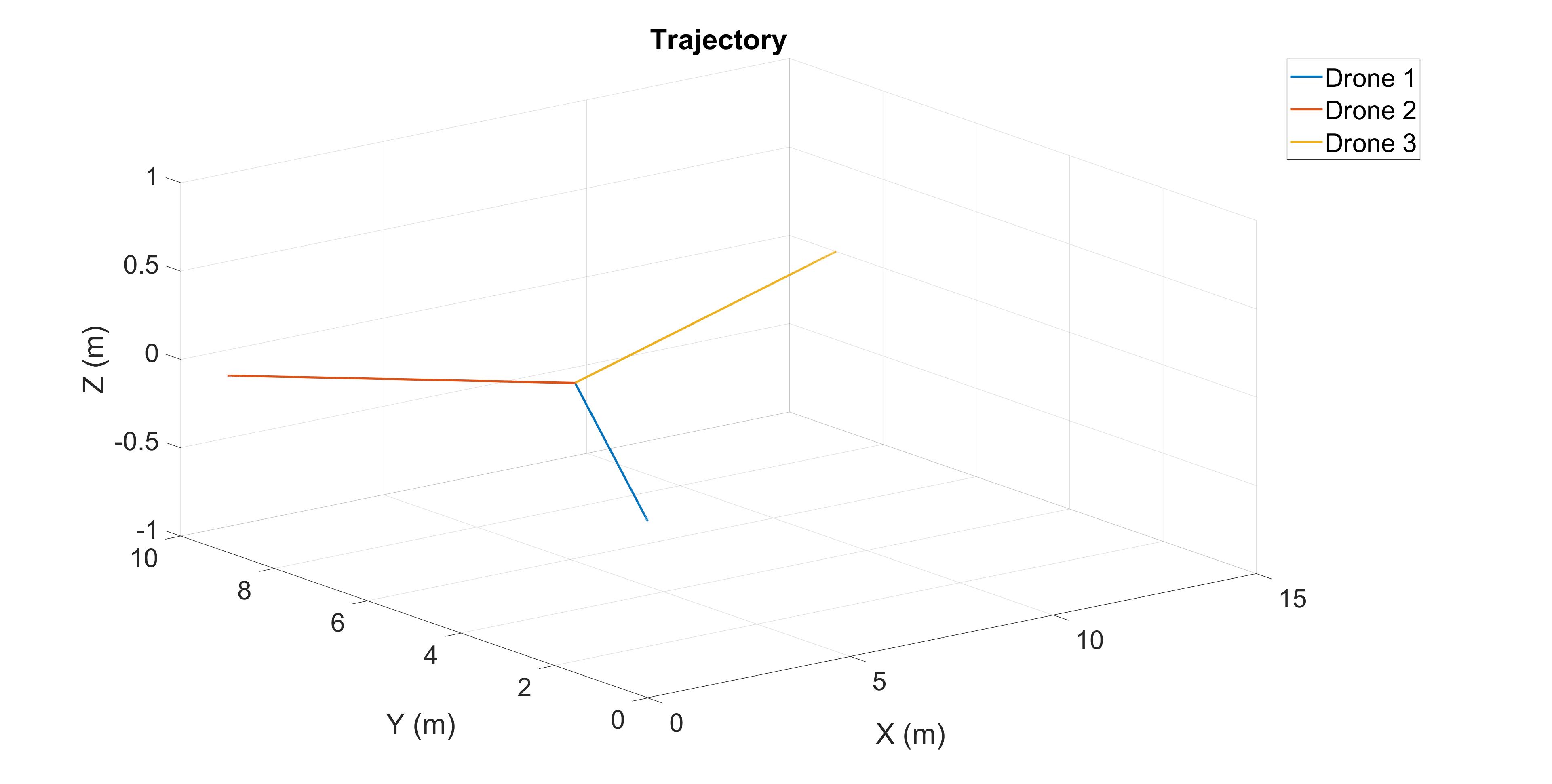}
    \caption{Three drones start at position $(0,0,0)$, $(0,9,0)$ and $(15,9,0)$ with initial yaw angles $0$, $-\pi/4$, and $\pi/2$.  Their trajectories to the rendezvous position are shown.}
    \label{fig: rendezvous traj}
\end{figure}

There are, however, two different methods for parametrizing these line segments and hence planning the drones' flights.  First, we could use the methods from Section \ref{section multi-agents}, obtaining trajectories which exponentially approach the rendezvous position.  In this setting, the drones' orientations are ignored completely, as are all dynamics, giving a simple yet unrealistic path.  

A main objective of this project was to investigate a second method: implement the path-planning from Section \ref{subsec: planning} to find controls bringing the drones along the desired line segments in a manner that respects the realistic dynamics from Section \ref{section-dynamics}.  This requires each drone beginning its flight by executing a pure yaw rotation to orient its body $x$-axis toward the rendezvous position, then fly straight in this direction.  See Figure \ref{fig: rendezvous yaw} to see the yaw for each drone over time; note that the total flight times differ for each drone.  Also note that for each drone, the roll over time is identically zero, but the pitch is non-zero in order to produce lateral displacement (compare with Figure \ref{fig: up and over config}).

\begin{figure}
    \centering
    \includegraphics[scale=0.07]{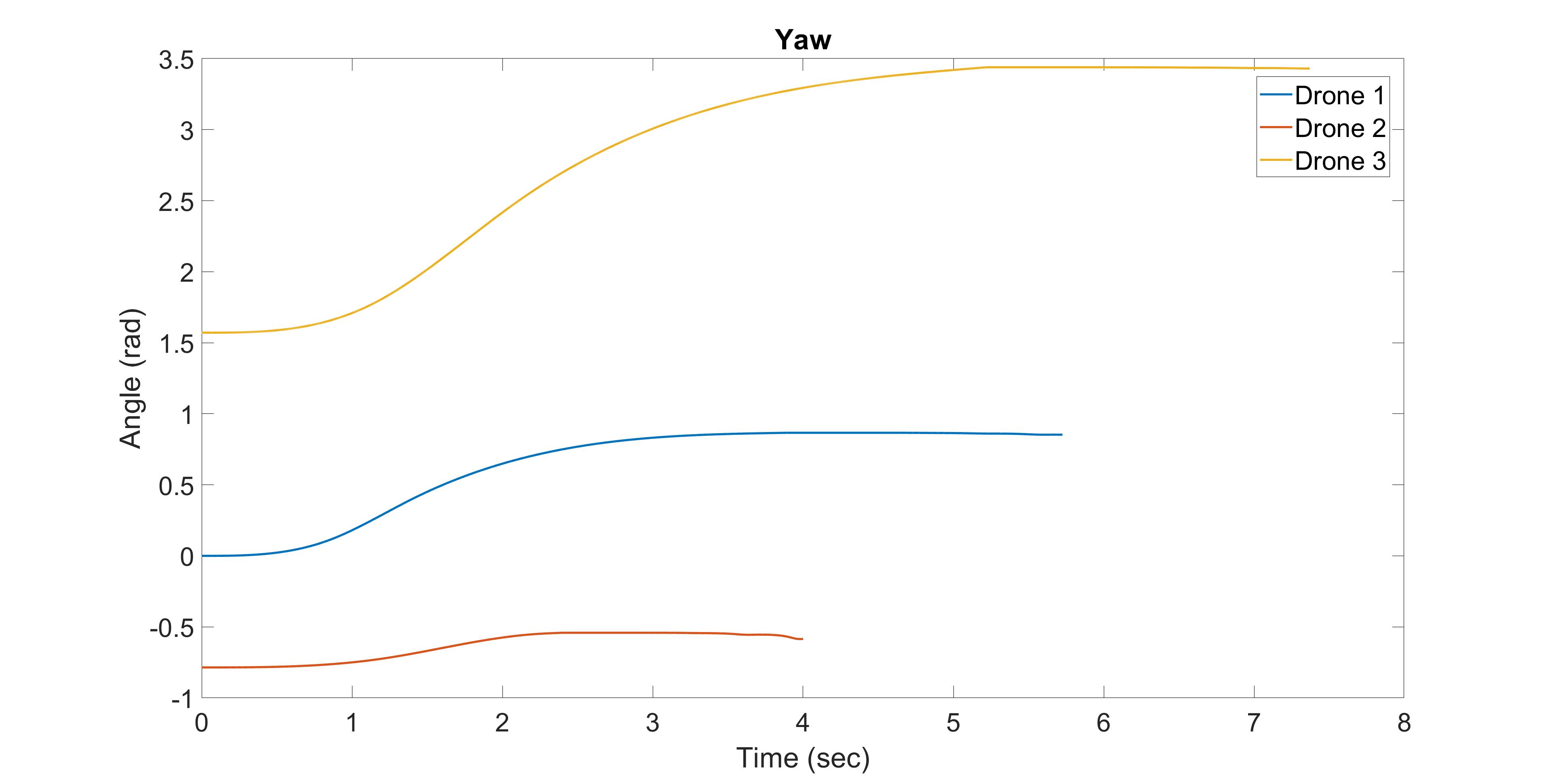}
    \caption{The yaw $\psi$ over time for each drone.  Total flight times differ for each drone.}
    \label{fig: rendezvous yaw}
\end{figure}

In order to create the desired trajectories in accordance with our equations of motion, we utilize Lemmas \ref{lemma: yaw motion} and \ref{lemma: x-axis motion}.  The resulting controls -- the four motor speeds for each of the three drones -- are shown in Figure \ref{fig: rendezvous controls}, as is the total thrust for each drone.  The parameters used in this simulation were taken from \cite{Bouadi2}, and appear in Table \ref{tab:params}.

\begin{table}
\caption{Parameters used in simulations}
\def\arraystretch{1.2}
  \begin{tabular}{ |  c | c |c |}
    \hline
    Constant & Symbol & Value \\ \hline 
    drone mass & $m$ & 0.468 \\ \hline 
    drone inertia  & $J$  & diag($(3.8278,3.8288,7.6566)\cdot10^{-3}$) \\ \hline 
   rotor inertia & $J_r$  & diag$(0, 0,2.8385\cdot 10^{-5})$  \\ \hline
       distance to rotor  & $d$  & 0.25  \\ \hline 
   thrust coefficient & $K_r$  & $2.9842\cdot 10^{-5}$ \\ \hline 
     translational drag & $C_D$ & $(5.5670,5.5670,6.3540)\cdot10^{-4}$  \\ \hline 
   rotational drag & $C_{\tau}$ & $(5.5670,5.5670,6.3540)\cdot 10^{-4}$  \\ \hline 
   propeller drag & $K_d$ & $3.2320\cdot 10^{-7}$ \\ \hline 
  \end{tabular}
  \label{tab:params}
\end{table}

 \begin{figure}
    \centering
    \includegraphics[scale=0.07]{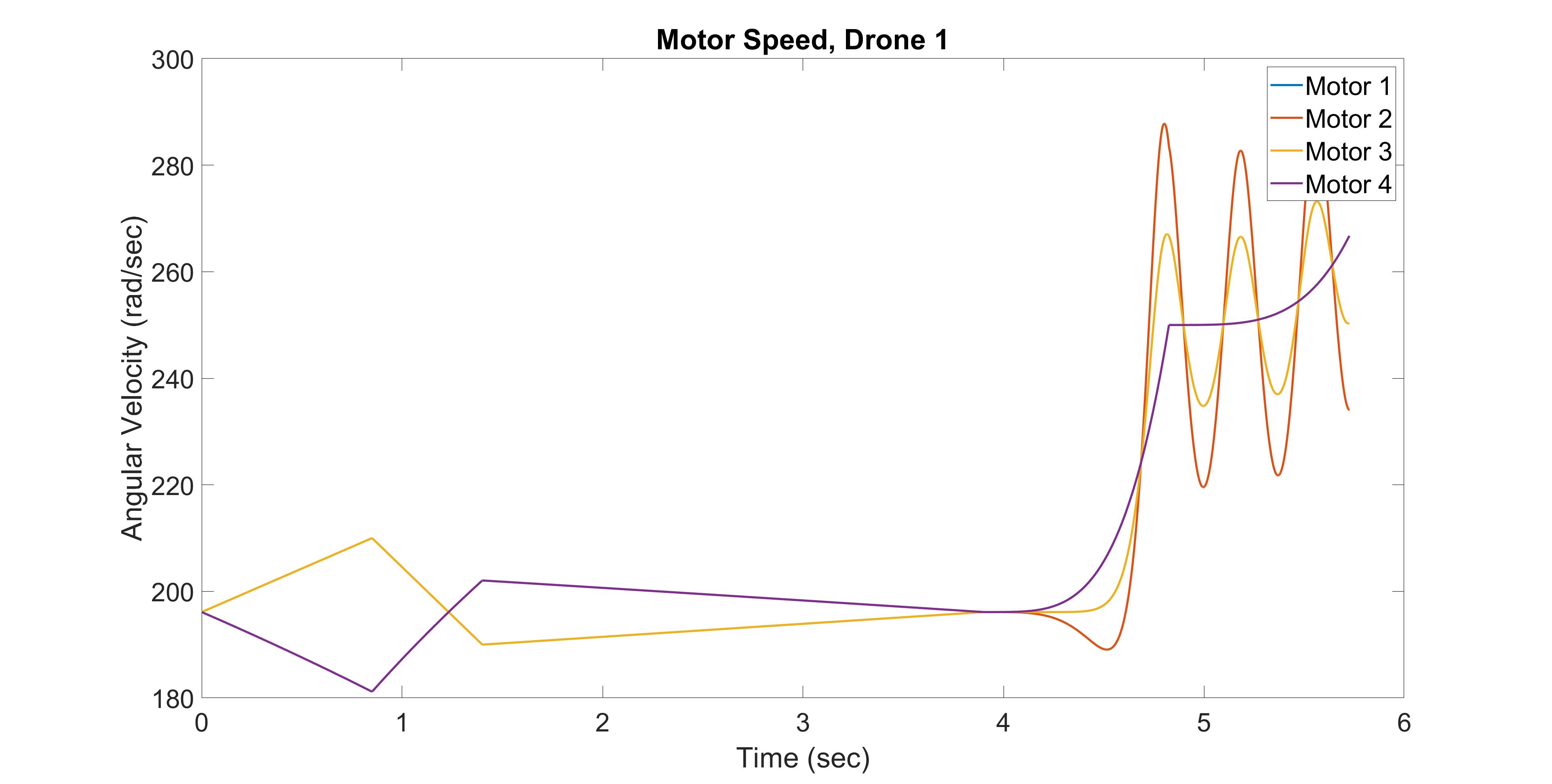}
    \includegraphics[scale=0.07]{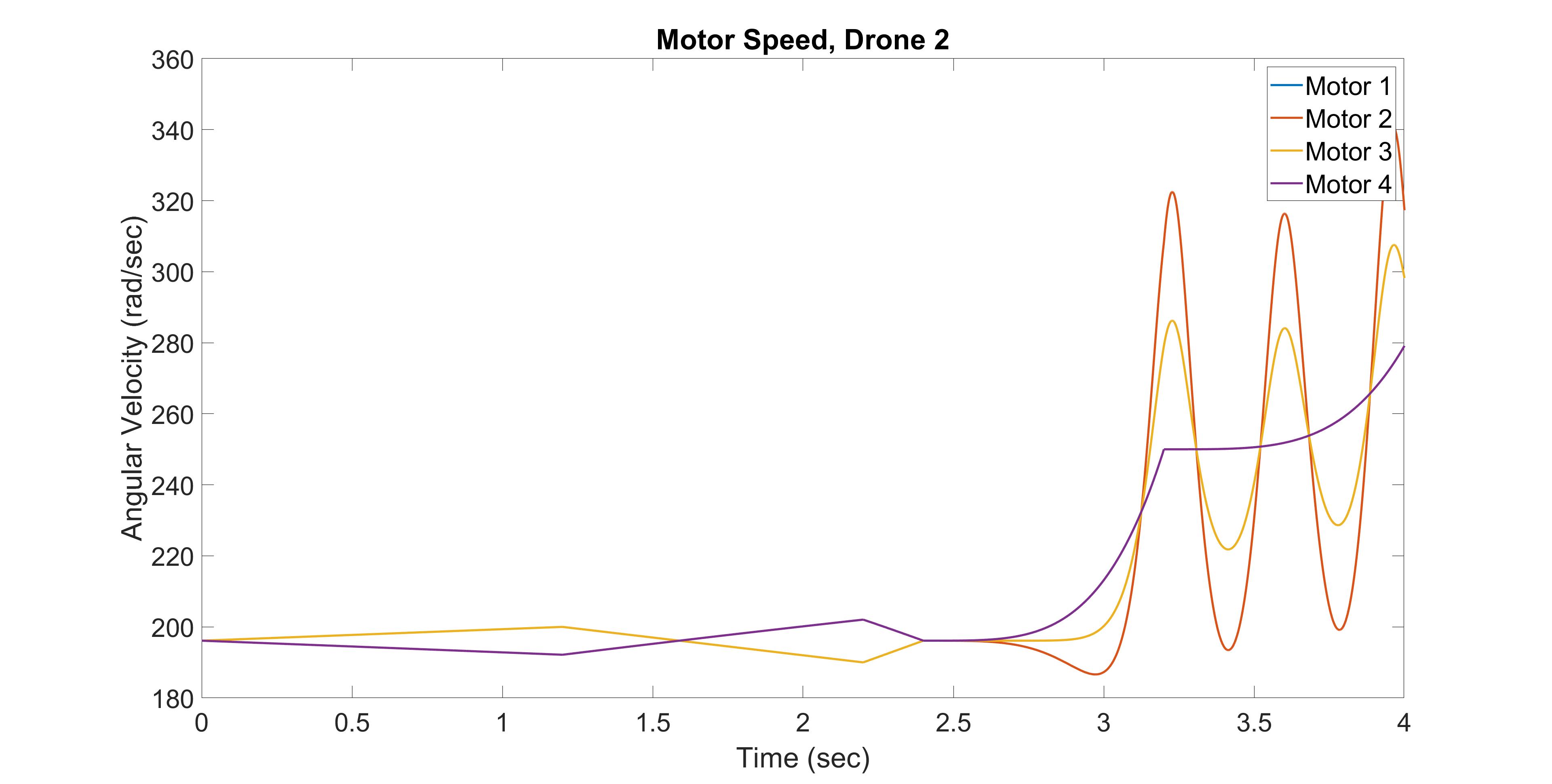}
    \includegraphics[scale=0.07]{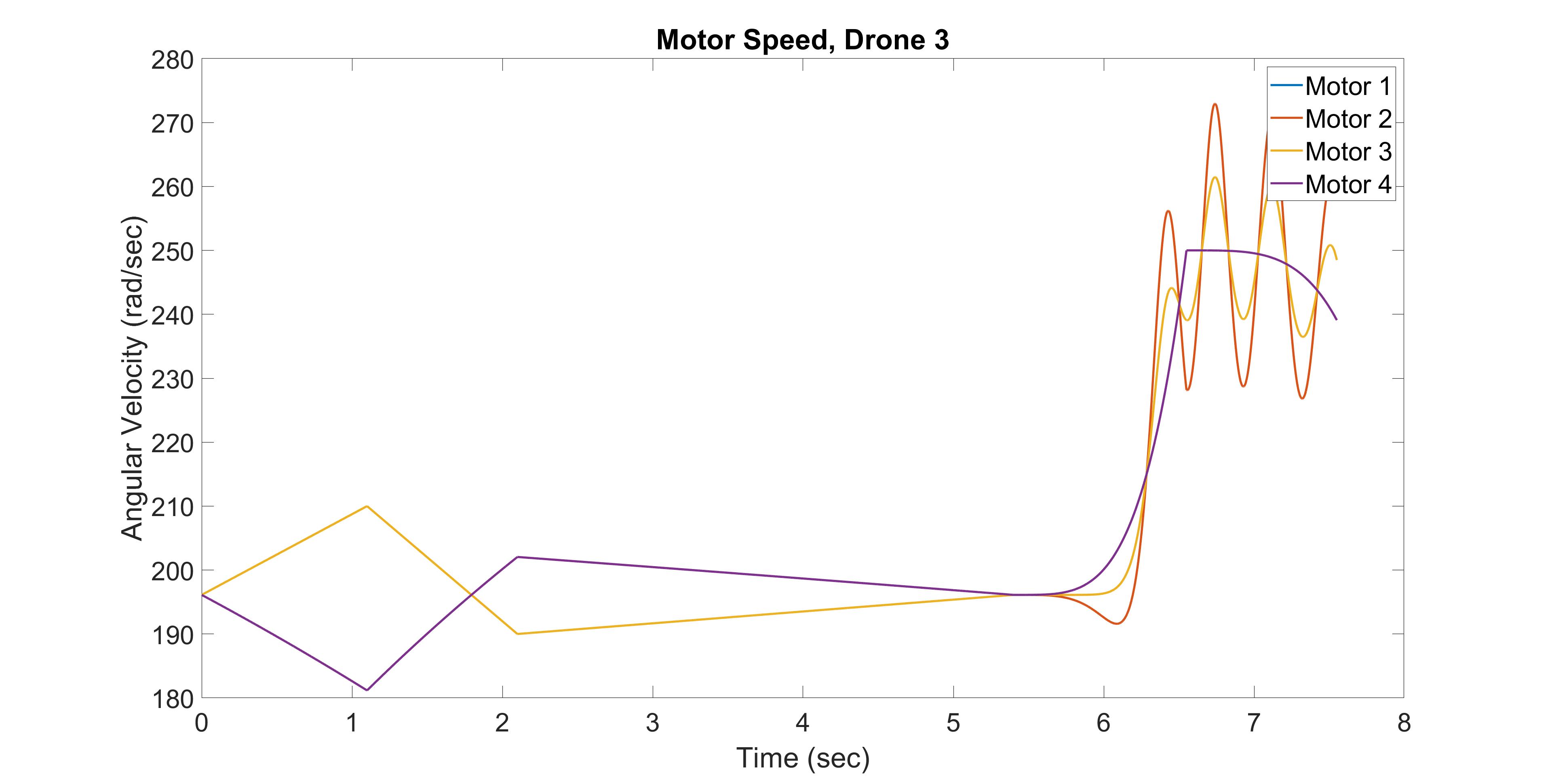}
    \includegraphics[scale=0.07]{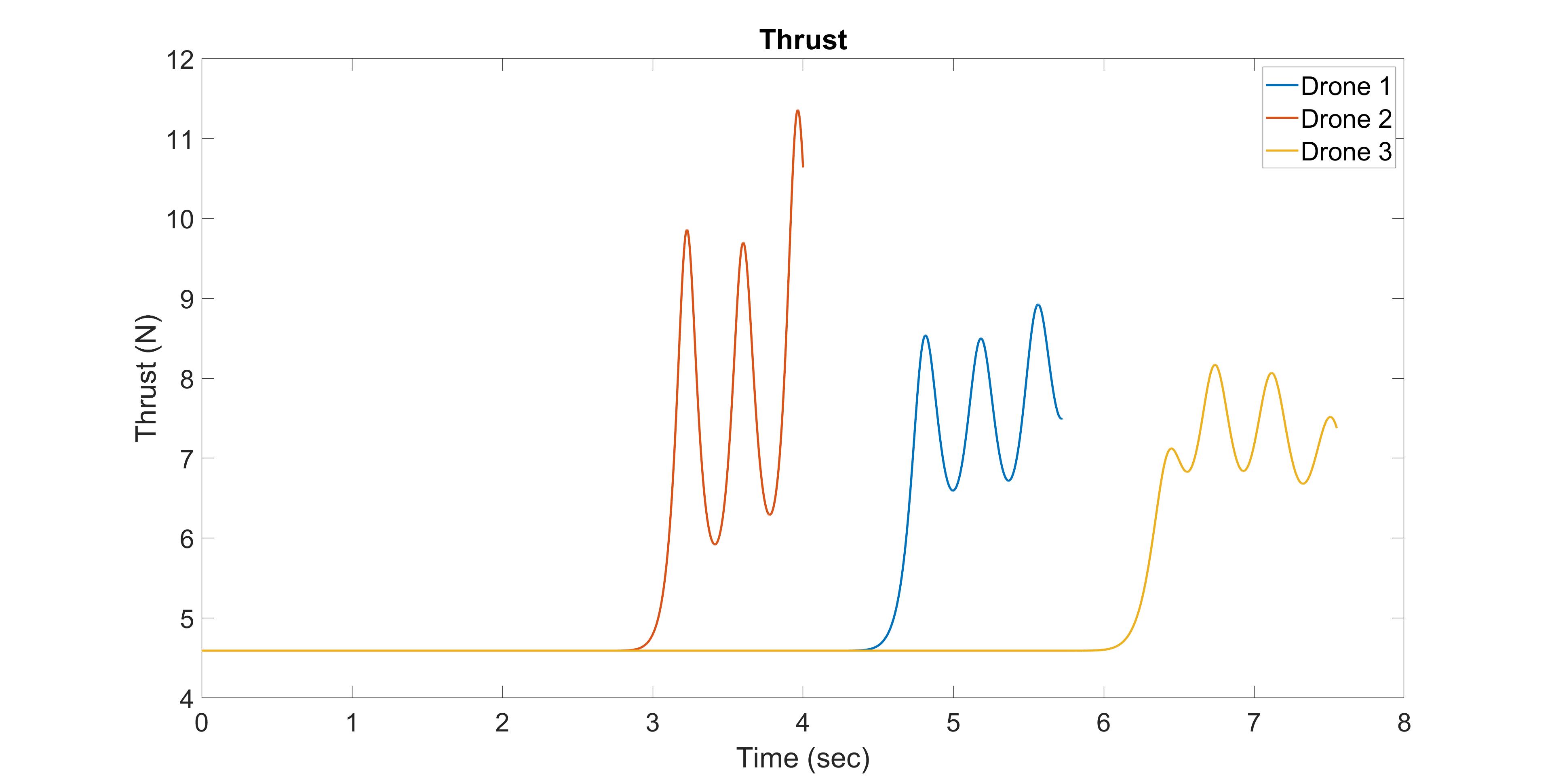}
    \caption{The controls used to produce the motions in Figures \ref{fig: rendezvous traj} and \ref{fig: rendezvous yaw}.  We show the four motor speeds and total thrust for each drone as functions of time.}
    \label{fig: rendezvous controls}
\end{figure}
 
 Finally, in Figure \ref{fig: rendezvous compare} we compare the paths generated by the two methods above. 
 The solid curves represent the simple but unrealistic paths derived in Section \ref{section multi-agents}.
 The dashed curves represent the dynamically sound paths in Figures \ref{fig: rendezvous yaw} and \ref{fig: rendezvous controls}.  
 Both methods trace the same curves shown in Figure \ref{fig: rendezvous traj}, but differ in their parametrizations.
 The top plot shows the $x$-coordinate over time; the solid curves for drones 1 and 2 coincide as both start with an initial $x$-value of 0.
 The bottom plot shows the $y$-coordinate over time; the solid curves for drones 2 and 3 coincide as both start with an initial $y$-value of 9.
 Note that the solid curves approach (but do not reach) the rendezvous position much faster than the dashed curves as they are unconstrained by orientation or realistic acceleration.
To move from the simulation environment to a more realistic experiment with actual drones would require a more careful consideration of the observer design as in \cite{DistribControl}.

\begin{figure}
\includegraphics[width=.49\linewidth]{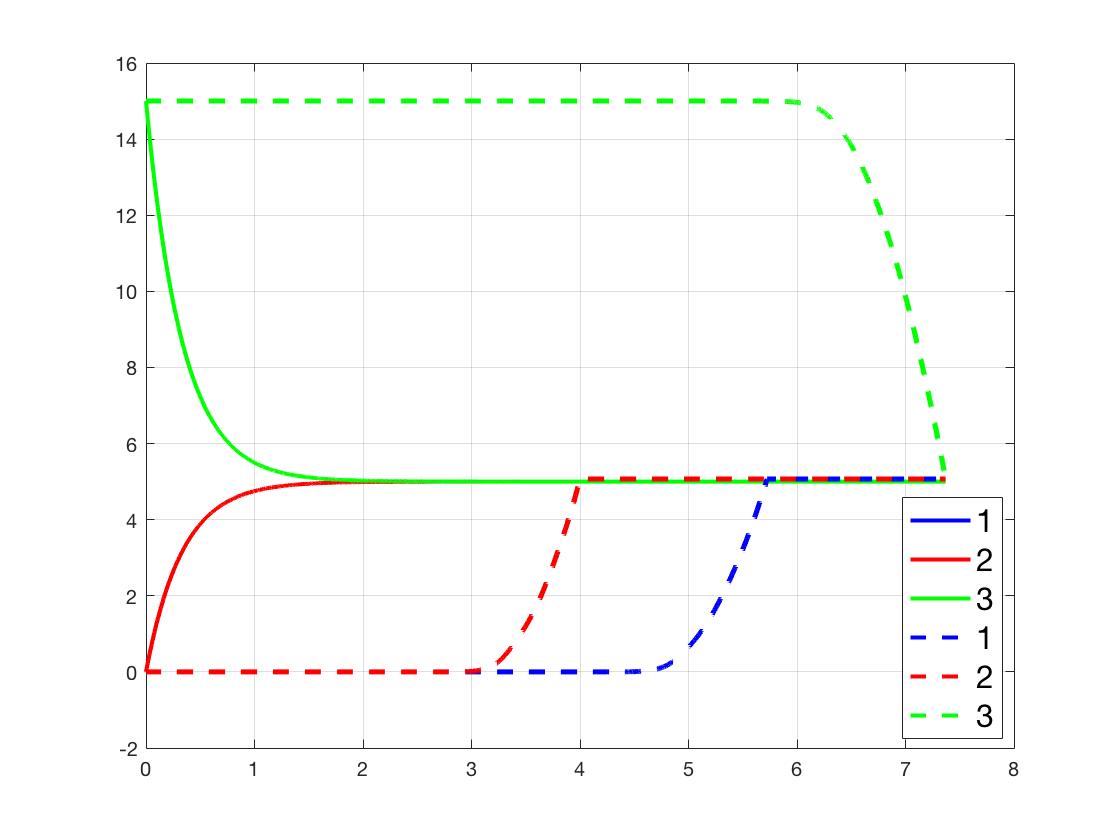}
\includegraphics[width=.49\linewidth]{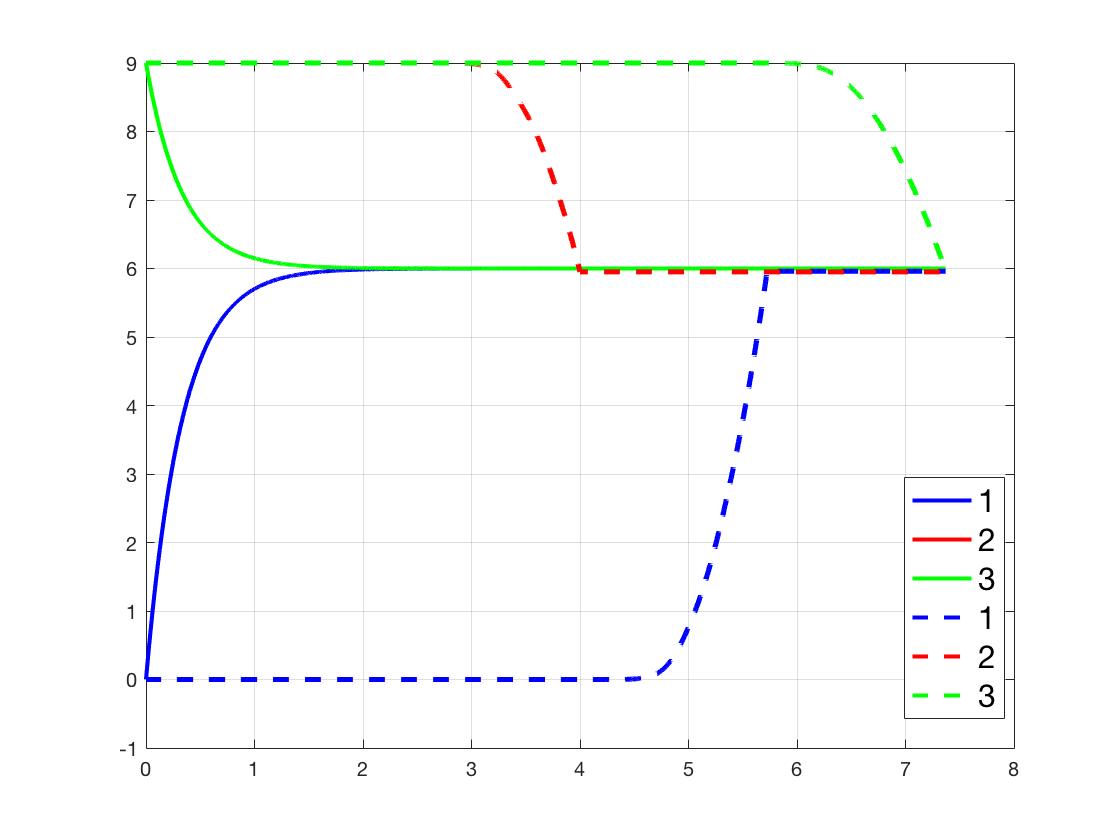}
\caption{Comparison of two methods of motion planning.  Dashed curves include dynamics; solid curves do not.  Left: $x$-coordinate over time.  Right: $y$-coordinate over time.}
\label{fig: rendezvous compare}
\end{figure} 

\section{Conclusion}
The primary goals of the paper were twofold. First, we introduce the agreement protocol from a graph theoretic viewpoint and apply it to a rendezvous mission for a set of particle agents without orientation. Second, we derive the equations of motion for a quadcopter to take into account its orientation and dynamics. We then apply the quadcopter dynamics to a rendezvous mission and compare the trajectories with the ones from the network multi-agents algorithm. This is a first step towards developing more complex multi-agent systems for quadcopters. A next step will be to analyze the decoupling vector fields for the quadcopter from its affine connection control system formulation and use these trajectories to approximate geometric trajectories obtained from the multi-agents network approach. We focused here on the agreement protocol, but ongoing work shows that it can be extended to formation flying for multi-agents and more complex missions, such as search and rescue with a large team of quadcopters including leaders and followers. We are also planning to run real drone flights to demonstrate the applicability of our techniques. 



\medskip
Received xxxx 20xx; revised xxxx 20xx.
\medskip

\end{document}